\documentclass[11pt]{article}
\usepackage[utf8]{inputenc}
\usepackage{lmodern}
\usepackage{subfiles}
\usepackage{enumitem}
	\setenumerate{label={\normalfont(\roman*)}, itemsep=0em} 
        
\usepackage{amsfonts}
\usepackage{amsthm}
\usepackage{amsmath}
\usepackage{amssymb}
\usepackage{amscd}
\usepackage{mathrsfs}
\usepackage{mathtools}
\usepackage{bm}
\usepackage{esint}

\usepackage[margin=3cm]{geometry}
\usepackage{setspace}
\usepackage{indentfirst}
\usepackage{graphicx}
\usepackage{graphics}
\usepackage{lscape}
\usepackage{tikz-cd}
\usepackage{color}
\usepackage{pict2e}
\usepackage{epic}
\usepackage{epstopdf}
\usepackage{titlesec, titlefoot}
	\titleformat{\section}[block]{\Large\bfseries\filcenter}{\thesection}{1em}{}
\usepackage{commath}
\usepackage{float}
\usepackage{caption}
\usepackage{etoolbox}
\usepackage[affil-it]{authblk}
\usepackage{combelow}

\usepackage[hidelinks, bookmarksdepth=3]{hyperref}
\hypersetup{bookmarksopen=true} 
\usepackage{hypcap}

\graphicspath{{./Pictures/}}
\allowdisplaybreaks

\usepackage{listings}
\usepackage{fancybox}

\theoremstyle{plain}
\newtheorem{bigthm}{Theorem}

\newtheorem{bigcor}[bigthm]{Corollary}

\renewcommand*\thesection{\arabic{section}}
\numberwithin{equation}{section} 

\theoremstyle{plain}
\newtheorem{thm}{Theorem}
\newtheorem*{thm*}{Theorem}
\newtheorem{lemma}[thm]{Lemma}
\newtheorem{prop}[thm]{Proposition}
\newtheorem{cor}[thm]{Corollary}
\numberwithin{thm}{section} 

\theoremstyle{definition}
\newtheorem{ndef}[thm]{Definition}
\newtheorem{ex}[thm]{Example}
\newtheorem{question}[thm]{Question}

\newtheorem{remark}[thm]{Remark}
\newtheorem{conj}[thm]{Conjecture}

\newcommand{\thistheoremname}{}
\newtheorem{genericthm}[equation]{\thistheoremname}

\newcommand{\thistheoremnames}{}
\newtheorem*{genericthms}{\thistheoremnames}
\newenvironment{para*}[1]
  {\renewcommand{\thistheoremnames}{#1}%
   \begin{genericthms}}
  {\end{genericthms}}

\expandafter\let\expandafter\oldproof\csname\string\proof\endcsname
\let\oldendproof\endproof
\renewenvironment{proof}[1][\proofname]{%
  \oldproof[\upshape \bfseries #1]%
}{\oldendproof}

\makeatletter
\def\@makechapterhead#1{%
  \vspace*{50\p@}%
  {\parindent \z@ \raggedright \normalfont
    \interlinepenalty\@M
    \Huge\bfseries  \thechapter.\quad #1\par\nobreak
    \vskip 40\p@
  }}
\makeatother

\usepackage{tocloft}
\cftsetindents{section}{0em}{2em}
\cftsetindents{subsection}{0em}{2em}
\def \a{\alpha}
\def \R {\mathbb{R}}

\def \C{\mathbb{C}}

\def \N{\mathbb{N}}
\def \D{\textup{D}}
\def \J{\textup{J}}

\def \T{\mathbb{T}}

\def \e{\varepsilon}
\def \d{\,\textup{d}}

\def \exc{\backslash}
\def \p{\partial}

\def \mc{\mathcal}
\def \mb{\mathbb}

\def \hs{\hspace{0.5cm}}

\def \tp{\textup}

\begin{document}

\title{\textbf{Nonlinear open mapping principles,\\ with applications to the Jacobian equation\\ and other scale-invariant PDEs}}

\author[1]{{\Large Andr\'e Guerra}}
\author[1]{{\Large Lukas Koch}}
\author[2]{{\Large Sauli Lindberg}}

\affil[1]{\small University of Oxford, Andrew Wiles Building Woodstock Rd, Oxford OX2 6GG, United Kingdom 
\protect \\
  {\tt{\{andre.guerra, lukas.koch\}@maths.ox.ac.uk}}
  \vspace{1em} \ }

\affil[2]{\small 
Aalto University, Department of Mathematics and Systems Analysis, P.O. Box 11100, FI-00076 Aalto, Finland 
\protect\\
  {\tt{sauli.lindberg@aalto.fi}} \ }

\date{}

\maketitle

\begin{abstract}
For a nonlinear operator $T$ satisfying certain structural assumptions, our main theorem states that the following claims are equivalent: i) $T$ is surjective, ii) $T$ is open at zero, and iii) $T$  has a bounded right inverse. The theorem applies to numerous scale-invariant PDEs in regularity regimes where the equations are stable under weak$^*$ convergence. Two particular examples we explore are the Jacobian equation and the equations of incompressible fluid flow.

For the Jacobian, it is a long standing open problem to decide whether it is onto between the critical Sobolev space and the Hardy space. Towards a negative answer, we show that, if the Jacobian is onto, then it suffices to rule out the existence of surprisingly well-behaved solutions. 

For the incompressible Euler equations, we show that, for any $p<\infty$, the set of initial data for which there are dissipative weak solutions in $L^p_t L^2_x$ is meagre in the space of solenoidal $L^2$ fields. Similar results hold for other equations of incompressible fluid dynamics.
\end{abstract}

\unmarkedfntext{
\hspace{-0.8cm}
\emph{Acknowledgments.} A.G. and L.K. were supported by the EPSRC [EP/L015811/1]. S.L. was supported by the AtMath Collaboration at the University of Helsinki and the ERC grant 834728-QUAMAP.
}

\begin{spacing}{0.89}
\setcounter{tocdepth}{1}
\tableofcontents
\end{spacing}

\section{Introduction}

The open mapping theorem is one of the cornerstones of functional analysis. When $X$ and $Y$ are Banach spaces and $L \colon X \to Y$ is a bounded \textit{linear} operator, it asserts the equivalence of the following two conditions:
\begin{enumerate}
\item Qualitative solvability: for all $f\in Y$ there is $u\in X$ with $Lu=f$, that is, $L(X) = Y$;
\item Quantitative solvability: for all $f \in Y$ there is $u \in X$ with $L u = f$ and $\Vert u \Vert_X \le C \Vert f \Vert_Y$.
\end{enumerate}
From a PDE perspective, the open mapping theorem justifies the method of a priori estimates \cite[§1.7]{Tao2010}. It is also a powerful tool in studying non-solvability of PDEs. 

For applications to nonlinear PDE, one would like to have an analogue of the open mapping theorem in the case where $L$ is replaced by a \textit{nonlinear} operator $T\colon X\to Y$. This is the main theme of the present paper.
The search for a nonlinear open mapping principle has been spurred by the following problem of \textsc{Rudin} \cite[page 67]{Rudin1969}:
\begin{question} \label{qu:Rudin}
If $X_1$, $X_2$ and $Y$ are Banach spaces and $T$ is a continuous bilinear map of $X_1 \times X_2$ onto $Y$, does it follow that $T$ is open at the origin?
\end{question}

Following \cite{Horowitz1975}, we say that \emph{the open mapping principle holds for $T$} if $T$ is open at the origin.
It is easy to see that, in general, $T$ is not open at all points. The origin plays a special role since, if $T$ is open at $0$, then by scaling one obtains quantitative solvability: for all $f \in Y$ there exist $u_i \in X_i$ such that $$T(u_1,u_2) = f \qquad \tp{ and } \qquad \Vert u_1 \Vert_{X_1}^2 + \Vert u_2 \Vert_{X_2}^2 \le C \Vert f \Vert_Y.$$

The answer to Question \ref{qu:Rudin} is negative. A first counter-example was given by \textsc{Cohen} in \cite{Cohen1974} and, shortly thereafter,  a much simpler one was found by \textsc{Horowitz}  \cite{Horowitz1975}, who considered the operator
$T \colon \R^3 \times \R^3 \to \R^4$ defined by 
 $$T(x,y) \equiv (x_1 y_1, x_1 y_2, x_1 y_3 + x_3 y_1 + x_2 y_2, x_3 y_2 + x_2 y_1),$$
 see also \cite{Bercovici1987,Dixon1988}. These examples suggest that there is little hope of having a general nonlinear open mapping principle. Nonetheless, as the main result of this paper, we find natural sets of conditions under which the open mapping principle holds.
 
\begin{thm*}[Rough version]
Consider a constant-coefficient system of PDEs with scaling symmetries, posed either over $\R^n$ or $\R^n\times [0,+\infty)$, which moreover is preserved by weak$^*$-convergence. The solution-to-data operator $T$ has a nonlinear open mapping principle.
\end{thm*}

For a precise version of the theorem we refer the reader to Theorem \ref{thm:abstractopenmappingprinciple2}. Besides being genuinely nonlinear, this result has the important advantage of not requiring the operator $T$ to be defined over a vector space: as we will see below, this flexibility is very useful for applications to evolutionary PDEs. 
 
The rest of this introduction is structured as follows. In the next subsection we present a precise version of the above theorem in a simple but important special case and then, in Section \ref{sec:introJacobian}, we apply this version to the Jacobian equation. In Section  \ref{sec:generalisedopenmap} we discuss more general open mapping principles, suited for scale-invariant equations, and in Section \ref{sec:applicationstofluids} we expand on the applications of these generalised versions to the equations of incompressible fluid dynamics.

\subsection{A first nonlinear open mapping principle}
\label{sec:thmA}

In this subsection we focus on a simple-to-state nonlinear open mapping principle.

\begin{bigthm} \label{thm:abstractopenmappingprinciple}
Let $X$ and  $Y$ be Banach spaces such that $\mb B_{X^*}$ is sequentially weak$^*$ compact. We make the following assumptions:

\begin{enumerate}[label={\normalfont (A\arabic*)}]
\item\label{it:weakcontinuityintro} $T\colon X^*\to Y^*$ is a weak$^*$-to-weak$^*$ sequentially continuous operator.

\item \label{it:homogeneityintro} $T(a u) = a^s T(u)$ for all $a > 0$ and $u \in X^*$, where $s > 0$.

\item \label{it:isometriesintro} For $k\in \N$ there are isometric isomorphisms $\sigma_k^{X^*}\colon X^*\to X^*$, $\sigma_k^{Y^*}\colon Y^*\to Y^*$ such that 
\begin{align*}
 T \circ \sigma_k^{X^*} = \sigma_k^{Y^*} \circ T \qquad \text{for all } k \in \N, \qquad
\sigma_k^{Y^*} f \overset{*}{\rightharpoonup} 0 \qquad \text{for all } f \in Y^*.
\end{align*}
\end{enumerate}

\noindent Then the following conditions are equivalent:
\begin{enumerate}

\item $T$ is onto: $T(X^*) = Y^*$;

\item $T$ is open at the origin;

\item For every $f \in Y^*$ there exists $u \in X^*$ such that
$T u = f$ and  $\Vert u \Vert_{X^*}^s \le C \Vert f \Vert_{Y^*}.$
\end{enumerate}
\end{bigthm}

Here and in the sequel $\mb B_{X^*}$ denotes the unit ball of $X^*$. The hypothesis that $\mb B_{X^*}$ is sequentially weak$^*$ compact holds, for instance, whenever $X$ is separable or reflexive. 
In general, the weak$^*$ topology on a dual space depends on the specific choice of predual: for instance, the spaces $c$ and $c_0$ induce different weak$^*$ topologies in their dual space, $\ell^1$. In fact it is only in this way that the spaces $X$ and $Y$ play a role in the statement of Theorem \ref{thm:abstractopenmappingprinciple}. We also note that the simple example $T\colon \R\times L^2(\R^n)\to L^2(\R^n), (t, f)\mapsto t f$ shows that there are operators satisfying the assumptions of Theorem \ref{thm:abstractopenmappingprinciple} which are not open at all points \cite{Downey2006}.

The assumption \ref{it:weakcontinuityintro} is not always necessary, but it holds automatically in finite dimensional examples. An infinite dimensional case where it is not needed is the multiplication operator $(f,g) \mapsto f g \colon L^p \times L^q \to L^r$, where $1/p+1/q=1/r$: this operator does not satisfy \ref{it:weakcontinuityintro}, although it verifies the open mapping principle \cite{Balcerzak2016, Balcerzak2013}. 

In Theorem \ref{thm:abstractopenmappingprinciple}, instead of considering multilinear operators as in Question \ref{qu:Rudin}, we consider the larger class of positively homogeneous operators, that is, operators satisfying \ref{it:homogeneityintro}. Nonetheless, many of the examples discussed in this paper are in fact multilinear.

The assumption \ref{it:isometriesintro} may look somewhat mysterious. However, as \textsc{Horowitz}'s example shows, it cannot be omitted.  
It should be thought of as \textit{generalised translation invariance} and indeed, when $T$ is a constant-coefficient partial differential operator and $X^*$ and $Y^*$ are function spaces on $\R^n$, natural choices of $\sigma_k^{X^*}$ and $\sigma_k^{Y^*}$ include translations 
$$\sigma_k^X u(x) \equiv u(x - k e) \tp{ and }\sigma_k^{Y^*} f(x) \equiv f(x - k e), \qquad \tp{where } e \in \R^n \setminus \{0\}.$$ We note that the condition \ref{it:isometriesintro} never holds if $Y$ is finite-dimensional and that moreover, when the target is two-dimensional, it is not needed: \textsc{Downey} has shown that, in this case, the answer to Question \ref{qu:Rudin} is positive \cite[Theorem 12]{Downey2006}.

In order for the reader to have a better grasp of the meaning of \ref{it:isometriesintro}, we briefly sketch the proof of Theorem \ref{thm:abstractopenmappingprinciple}, explaining the role of this assumption in it.
Suppose that, for all $f$ in some ball $B\subset Y^*$, one can solve the equation $Tu=f$. Since $T$ is weakly$^*$ continuous, it is not difficult to use the Baire Category Theorem to deduce that there is a sub-ball $B'\subseteq B$ such that one can actually solve $Tu=f$ quantitatively in $B'$: that is, there is a constant $C$ such that, for all $f\in B'$, there is $u\in X^*$ with $\Vert u \Vert_{X^*} \leq C$ and $Tu=f$. In other words: for weakly$^*$ continuous operators, qualitative  solvability implies quantitative local solvability \textit{somewhere}; in general, one cannot  specify the location of $B'$. The assumption \ref{it:isometriesintro} allows one to shift the centre of the ball $B'$ to the origin and, in combination with \ref{it:homogeneityintro}, it upgrades the previous local statement to a global version.

Compensated compactness theory \cite{Murat1981, Tartar1979} abounds with operators that satisfy the assumptions of Theorem \ref{thm:abstractopenmappingprinciple}. The most famous examples are the Jacobian, the Hessian and the div-curl product; see \cite{Coifman1992} for numerous examples and \cite{GuerraRaita2020, GuerraRaitaSchrecker2020} for a systematic study.  In fact, our motivation for Theorem \ref{thm:abstractopenmappingprinciple} came from considering the Jacobian operator and the spaces
\begin{equation}
X^*={\dot W}^{1,np}(\R^n,\R^n), \qquad Y^*=\mathscr H^p(\R^n), \qquad 1\leq p <\infty,
\label{eq:spaces}
\end{equation}
see Question \ref{qu:Lp} below. The real-variable Hardy space $\mathscr H^p(\R^n)$ is defined by fixing any $\Phi \in \mathscr{S}(\R^n)$ with $\int_{\R^n} \Phi(x) \d x \neq 0$, denoting $\Phi_t(x) \equiv \Phi(x/t)/t^n$ for all $(x,t) \in \R^n \times (0,\infty)$ and setting 
$$\mathscr{H}^p(\R^n) \equiv \big\{f \in \mathscr{S}'(\R^n): \Vert f \Vert_{\mathscr{H}^p} \equiv \Vert \sup_{t > 0} |f * \Phi_t(\cdot)|\Vert_{L^p} < \infty\big\}.$$ 
We refer the reader to the monograph \cite{Stein2016} for the theory of $\mathscr H^p(\R^n)$. Here we just note that
$\mathscr{H}^p(\R^n) = L^p(\R^n)$, with equivalent norms, whenever $p\in (1,\infty)$ and that moreover $\mathscr{H}^1(\R^n) \subsetneq \{f \in L^1(\R^n) \colon \int_{\R^n} f(x) \d x = 0\}$. Indeed, loosely speaking,  elements of $\mathscr H^1(\R^n)$ have an extra logarithm of integrability \cite{Stein1969}. In \eqref{eq:spaces}, the space ${\dot W}^{1,np}(\R^n,\R^n)$ is the usual homogeneous Sobolev space, seen as a Banach space, so that elements which differ by constants are identified.

\subsection{Applications to the Jacobian equation}
\label{sec:introJacobian}
 
As discussed in the last subsection, we are interested in  the Jacobian and, therefore, in the prescribed Jacobian equation
\begin{equation}
\label{eq:jac}
\tp{J}u\equiv \det{\D u}=f \qquad \tp{in } \R^n.
\end{equation}
This first-order equation appears naturally in Optimal Transport \cite{Brenier1991} and can be seen as the underdetermined analogue of the Monge--Ampère equation, see \cite{GuerraKochLindberg2020} for further discussion. It has a deep geometric content as, formally, one has the change of variables formula
\begin{equation}
\label{eq:changeofvariables}
\int_E \J u(x) \d x = \int_{\R^n} \# \big( u^{-1}(y)\cap E\big) \d y, \qquad E\subset \R^n \tp{ is measurable.}
\end{equation}
Thus, for a smooth solution of \eqref{eq:jac}, $f$ measures the size of its image, counted with multiplicity.

When applied to the Jacobian, Theorem \ref{thm:abstractopenmappingprinciple} reads as follows:
\begin{bigcor}\label{cor:openmapJacobian}
Fix $1\leq p <\infty$. The following statements are equivalent:
\begin{enumerate}
\item $\J\colon {\dot W}^{1,np}(\R^n,\R^n)\to \mathscr H^p(\R^n)$ is surjective;
\item there is a {\normalfont bounded} operator $E\colon {\dot W}^{1,np}(\R^n,\R^n)\to \mathscr H^p(\R^n)$ such that $\J \circ E=\tp{Id}$;
\item for all $f\in \mathscr H^p(\R^n)$ there is $u\in {\dot W}^{1,np}(\R^n,\R^n)$ such that $\J u = f$ and 
\begin{equation}
\label{eq:aprioriestimateintro}
\Vert \D u \Vert_{L^{np}(\R^n)}^n \lesssim \Vert f \Vert_{\mathscr H^p(\R^n)}.
\end{equation}
\end{enumerate}
\end{bigcor}

We briefly describe the motivation behind Corollary \ref{cor:openmapJacobian}. The Jacobian benefits from improved integrability: after a remarkable result of \textsc{M\"uller} \cite{Muller1990}, 
\textsc{Coifman}, \textsc{Lions}, \textsc{Meyer} and \textsc{Semmes} proved in \cite{Coifman1993} that
$$u\in {\dot W}^{1,n}(\R^n,\R^n)\quad \implies \quad \J u\in \mathscr H^1(\R^n)$$
and that various other compensated compactness quantities also enjoy $\mathscr{H}^1$ integrability. \textsc{Coifman \& al.}\ proceeded to ask whether compensated compactness quantities are surjective into the Hardy space $\mathscr{H}^1(\R^n)$. The third author showed in \cite{Lindberg2017} that the question must be formulated in terms of homogeneous Sobolev spaces, if it is to have a positive answer.
The following problem remains open:

\begin{question}\label{qu:Lp}
For $p\in [1,\infty)$ and  $f\in \mathscr{H}^p(\R^n)$, is there $u\in \dot{W}^{1,np}(\R^n,\R^n)$ solving \eqref{eq:jac}?
\end{question}

We note that the case $p>1$ was posed by \textsc{Iwaniec}, see also Conjecture \ref{conj:iwaniec} below for a stronger formulation of Question \ref{qu:Lp}.

It is very difficult to predict the answer to Question \ref{qu:Lp}, and in both directions, the currently available evidence is only tentative. A positive answer  seems currently out of reach, as there is no systematic way of building solutions to \eqref{eq:jac} for a general discontinuous $f$, although see \cite{Riviere1996} for some endpoint cases. This is in contrast with the case of H\"older continuous $f$, where there is a well-posedness theory which goes back to the works of \textsc{Dacorogna} and \textsc{Moser} \cite{Dacorogna1990, Moser1965}, see also \cite{Csato2012, Ye1994} and the references therein. 

In this subsection, we focus on progress towards a negative answer to Question \ref{qu:Lp}.
Our motivation for doing so comes, to some extent, from our results in \cite{GuerraKochLindberg2020}. There, we proved that a variant of Question \ref{qu:Lp}---obtained by replacing $\R^n$ with a bounded, smooth domain $\Omega$ and additionally imposing Dirichlet boundary conditions on the solutions---is false. 

The main difficulty in proving non-existence of solutions to \eqref{eq:jac} is the \textit{underdetermined} nature of the equation: it implies that there is a multitude of possible solutions to rule out. This is the main reason why Question \ref{qu:Lp} is much harder than its analogue on a bounded domain, as the lack of boundary conditions makes the problem even more underdetermined. Indeed, the space ${\dot W}^{1,np}(\R^n,\R^n)$ is extremely large, and so there is an abundance of possible solutions to consider, and moreover it contains many poorly-behaved maps, especially when $p=1$. We now make precise this last point.

When $p=1$, there are continuous maps in ${\dot W}^{1,n}(\R^n,\R^n)$ which do not satisfy the change of variables formula \eqref{eq:changeofvariables}, as they map a null set $E$ into a set of positive measure, and hence
$$0=\int_E f\d x < |u(E)|\leq \int_{\R^n} \# \big( u^{-1}(y)\cap E\big) \d y.$$
It is thus possible for \eqref{eq:jac} to hold a.e.\ in $\R^n$, and hence also in the sense of distributions, and yet for its geometric information to be completely lost! Maps as above are said to violate the \textit{Lusin (N) property}; their existence is classical and goes back to the work of \textsc{Cesari} \cite{Cesari1942}, see also \cite{Maly1995} for a more refined version.

We are interested in considering \textit{admissible solutions} of \eqref{eq:jac}, i.e.\ solutions for which the geometric information of the equation is preserved; here our choice of terminology is inspired by the fluid dynamics literature. Corollary \ref{cor:openmapJacobian} yields the quite surprising fact that
\begin{quote}
\textit{the existence of rough solutions implies the existence of admissible solutions}.
\end{quote}
This fact is made precise in the following result:

\begin{bigthm}\label{bigthm:nicesols}
Let $\Omega\subset \R^n$ be a bounded open set and take $f\in \mathscr H^1(\R^n)$ such that $f\geq 0$ in $\Omega$. Assume that $\J\colon {\dot W}^{1,n}(\R^n,\R^n)\to \mathscr H^1(\R^n)$ is onto. Then there is a solution $u\in {\dot W}^{1,n}(\R^n,\R^n)$ of \eqref{eq:jac} such that:
\begin{enumerate}
\item\label{it:continuityintro} $u$ is continuous in $\Omega$;
\item\label{it:Lusinintro} $u$ has the Lusin (N) property in $\Omega$;
\item $\int_{\R^n} |\D u|^n \d x \le C \Vert f \Vert_{\mathscr{H}^1}$ with $C > 0$ independent of $f$.
\end{enumerate} 
In particular, $u$ satisfies the change of variables formula \eqref{eq:changeofvariables}.

Moreover, if $n=2$ and there is an open set $\Omega'\subseteq \Omega$ with $f=0$ a.e.\ in $\Omega'$, then:
\begin{enumerate}[resume]
\item for any set $E\subset \Omega'$, we have $u(\p E)=u(\overline E)$;
\item\label{it:lastproperty} for $y\in u(\Omega')$, if $C$ denotes a connected component of $u^{-1}(y)\cap \Omega'$ then $C$ intersects $\p \Omega'$.
\end{enumerate}
\end{bigthm}

Theorem \ref{bigthm:nicesols} is proved through a regularisation argument: due to Corollary \ref{cor:openmapJacobian}, powerful tools from Geometric Function Theory become available.
In the supercritical regime $p>1$, the first part of Theorem \ref{bigthm:nicesols} holds automatically, although one can still use the a priori estimate \eqref{eq:aprioriestimateintro} to get solutions with additional structure, see Section \ref{sec:nicesols} for further details. 
The second part of Theorem \ref{bigthm:nicesols} also holds in any dimension if $p$ is taken to be sufficiently large.

Contrary to the approach discussed so far, it may be that the answer to Question \ref{qu:Lp} is positive. In this direction, \textsc{Iwaniec} suggested in \cite{Iwaniec1997} that one should prove \eqref{eq:aprioriestimateintro} for all \emph{$p$-energy minimisers}, that is, solutions of \eqref{eq:jac} which minimise $\int_{\R^n} |\D u|^{np} \d x$ under the constraint $\tp{J} u = f$. To achieve this goal, \textsc{Iwaniec} has proposed that one construct a \emph{Lagrange multiplier} for every $p$-energy minimiser. This turns out to be a very difficult task since the standard methods fail, see Section \ref{sec:submersion} for further discussion. Nonetheless, when $p=1$ and $n=2$, uniformly bounded Lagrange multipliers were constructed in \cite{Lindberg2015,Lindberg2020} for a large class of $p$-energy minimisers, which then automatically satisfy the  estimate \eqref{eq:aprioriestimateintro}; we also note that the methods of \cite{Lindberg2015,Lindberg2020} can be partly adapted to all the cases $n \ge 2$, $p \in [1,\infty)$. 

In  \cite{Iwaniec1997}, see also \cite{Bonami2007}, \textsc{Iwaniec}  went further than Question \ref{qu:Lp} and posed the following:
\begin{conj}\label{conj:iwaniec}
For $p\in [1,\infty)$, there is a continuous map
$E\colon \mathscr{H}^p(\R^n)\to \dot{W}^{1,np}(\R^n,\R^n)$ such that $\J \circ E = \tp{Id}$.
\end{conj}

Conjecture \ref{conj:iwaniec} proposes the existence of a continuous \textit{right inverse} for the Jacobian (abbreviated in this paragraph as r.i.). We have the following trivial chain of implications:
$$
\begin{array}{l}
\J \tp{ has a bounded,}\\
\tp{F-differentiable r.i.}
\end{array}
\implies 
\begin{array}{l}
\J \tp{ has a bounded,}\\
\tp{continuous r.i.}
\end{array}
\implies
\J \tp{ has a bounded r.i.}
\implies
\J \tp{ has a r.i.}
$$
Corollary \ref{cor:openmapJacobian} states the equivalence of the last two statements, while Corollary \ref{cor:JacobianNondifferentiable} below shows that the first statement is false. Finally, and concerning the second statement,  \textsc{Iwaniec} pointed out in~\cite{Iwaniec1997} that it is plausible that $p$-energy minimisers are unique up to rotations: if so, this would pave a road towards a solution of Conjecture \ref{conj:iwaniec}. However, and as further negative evidence towards Question \ref{qu:Lp} and Conjecture \ref{conj:iwaniec}, there are some data which admit uncountably many $p$-energy minimisers. This will be proved in our forthcoming work~\cite{GuerraKochLindberg2020c}.

\subsection{General nonlinear open mapping principles} 
\label{sec:generalisedopenmap}
In this subsection we formulate a more general open mapping principle. Precise statements, as well as multiple examples, can be found in Sections \ref{sec:A general nonlinear open mapping principle for scale-invariant problems} and \ref{sec:Non-surjectivity under incompatible scalings}, but Theorem \ref{bigthm:incompatiblescalings} below contains already a rough version of our main result.

In Theorem \ref{bigthm:incompatiblescalings},  the positive homogeneity assumption \ref{it:homogeneityintro} from Theorem \ref{thm:abstractopenmappingprinciple} is replaced by more general scaling symmetries; the assumption of translation invariance will be kept. To be precise, we assume that the equation ${T u(x,t) = f(x,t)}$, $(x,t) \in \R^n \times [0,\infty)$ is invariant under a one-parameter group of scalings
\begin{equation} \label{eq:scalingsintro}
\tau^{X^*}_\lambda[u](x,t) \equiv \frac{1}{\lambda^\alpha} u\left( \frac{x}{\lambda^\beta},\frac{t}{\lambda^\gamma} \right), \qquad \tau^{Y^*}_\lambda[f](x,t) \equiv \frac{1}{\lambda^\delta} f \left( \frac{x}{\lambda^\beta},\frac{t}{\lambda^\gamma} \right),
\end{equation}
where $\alpha,\beta,\gamma,\delta \in \R$ are fixed and the group parameter is $\lambda > 0$. In our typical applications, $f$ is the initial datum of a Cauchy problem and $T$ is the solution-to-datum map.

Invariance under translations and scalings is an ubiquitous feature of physical processes. It expresses the \emph{covariance principle} that the solutions of a PDE representing a physical phenomenon should not have a form which depends on the location of the observer or the units that the observer is using to measure the system \cite{Budd2001}. For the computation of the symmetry groups of several representative PDEs we refer to \cite[\textsection 2.4]{Olver1993} and for the general role of scaling symmetries in physics and other sciences to \cite{Barenblatt2003}.

It often happens that a PDE has \textit{several scaling symmetries}. For instance, the positive $n$-homogeneity of the Jacobian operator can be expressed as symmetry of the equation $\tp{J} u = f$ under the scaling $u_\lambda = \lambda u$, $f_\lambda = \lambda^n f$ for all $\lambda > 0$, but the Jacobian equation also has the scaling symmetry $u_\lambda(x) = \lambda u(x/\lambda)$, $f_\lambda(x) = f(x/\lambda)$. For another example, concerning the incompressible Euler equations, see \eqref{eq:Euler1}--\eqref{eq:Eulerscalings}.  An important theme in this work is that, whenever a PDE has several scaling symmetries, these symmetries must be \textit{compatible} in order for the equation to be solvable for all data. 

The next result encapsulates the two previous points: many scale-invariant PDEs satisfy a nonlinear open mapping principle and, for the equation to be solvable, the associated scalings need to be compatible.

\begin{bigthm}[Rough version]\label{bigthm:incompatiblescalings}
Consider a constant-coefficient system of PDEs, posed over $\R^n\times [0,\infty)$, which moreover is preserved under weak$^*$ convergence. Let $T$ be the solution-to-datum operator associated with the PDE.

Suppose that the equation $T u=f$ is invariant under the scalings \eqref{eq:scalingsintro} and that the solutions and the data lie in \emph{homogeneous} function spaces, which satisfy, for some $r,s\in \R$,
\[\Vert \tau^{X^*}_\lambda[u] \Vert_{X^*} \equiv \lambda^r \Vert u \Vert_{X^*}, \qquad \Vert \tau^{Y^*}_\lambda[f] \Vert_{Y^*} \equiv \lambda^s \Vert f \Vert_{Y^*},\qquad
\text{ where } r s >0.\]
The following statements are then equivalent:
\begin{enumerate}
\item For all $f \in Y^*$ there is $u \in X^*$ with $T u = f$;

\item For all $f \in Y^*$ there is $u \in X^*$ with $T u = f$ and $\Vert u \Vert_{X^*}^{s/r} \le C \Vert f \Vert_{Y^*}$.
\end{enumerate}
Moreover, suppose that $T$ is invariant under another pair of scalings $\tilde \tau_\lambda^{X^*}, \tilde \tau_\lambda^{Y^*}$, which satisfy
\[\Vert \tilde\tau^{X^*}_\lambda[u] \Vert_{X^*} \equiv \lambda^{\tilde r} \Vert u \Vert_{X^*}, \qquad \Vert \tilde \tau^{Y^*}_\lambda[f] \Vert_{Y^*} \equiv \lambda^{\tilde s} \Vert f \Vert_{Y^*},\qquad
\text{ where } \tilde r \tilde s >0.\]
Then solvability of the equation $Tu=f$ requires compatibility of the scalings, i.e.\
$$T \text{ is surjective } \quad \implies \quad r/s= \tilde r/\tilde s.$$
\end{bigthm}

The first part of Theorem \ref{bigthm:incompatiblescalings} yields a generalisation of Theorem \ref{thm:abstractopenmappingprinciple}. We note that the assumption that $r$ and $s$ have the same sign, i.e.\ $r s>0$, ensures that the norms in question are either \textit{subcritical} or \textit{supercritical}; the critical case $r = s = 0$ is beyond the scope of this work. We will expand on the second part of Theorem \ref{bigthm:incompatiblescalings} in the next subsection. Solvability and a priori estimates in inhomogeneous function spaces are treated in Corollary \ref{cor:abstractopenmappingprincipleintersections}. 

Concerning the hypothesis of stability under weak$^*$ convergence, we note that it is typically satisfied by solutions above a certain regularity threshold: for instance, both the Navier--Stokes equations and the cubic wave equation in $\R^3\times [0, +\infty)$ are preserved under weak$^*$ convergence in the corresponding energy spaces. For further details, see Section \ref{sec:Examples}. Moreover,  by considering a relaxed version of the PDE the assumption of stability under weak$^*$ convergence can sometimes be bypassed, an idea
which will be briefly discussed in the next subsection.

\subsection{Applications to the equations of incompressible fluid flow}
\label{sec:applicationstofluids}

The validity of the assumptions of Theorem \ref{bigthm:incompatiblescalings} is typically easy to check in practice. As such, the theorem gives rather immediate consequences on various physical PDEs. For further discussion of the merits and weaknesses of Theorem \ref{bigthm:incompatiblescalings}, see Section \ref{sec:Concluding discussion}.

In order to give a representative application of Theorem \ref{bigthm:incompatiblescalings} we consider energy-dissipating solutions of the incompressible Euler equations
\begin{align}
& \partial_t u + u \cdot \nabla u - \nabla P = 0, \label{eq:Euler1} \\
& \tp{div} \, u = 0, \label{eq:Euler2} \\
& u(\cdot,0) = u^0 \label{eq:Euler3}
\end{align}
in $\R^n \times [0,\infty)$, for $n \ge 2$. Note that \eqref{eq:Euler1}--\eqref{eq:Euler3} are invariant under  scalings of the form
\begin{equation} \label{eq:Eulerscalings}
u_\lambda(x,t) \equiv \frac{1}{\lambda^\alpha} u \left( \frac{x}{\lambda^\beta},\frac{t}{\lambda^{\alpha+\beta}} \right), \; u^0_\lambda(x,t) \equiv \frac{1}{\lambda^\alpha} u^0\left( \frac{x}{\lambda^\beta} \right), \; P_\lambda(x,t) \equiv \frac{1}{\lambda^{2\alpha}} P\left( \frac{x}{\lambda^\beta},\frac{t}{\lambda^{\alpha+\beta}} \right)
\end{equation}
for any $\a,\beta>0$.
An interesting connection between \eqref{eq:Euler1}--\eqref{eq:Euler2} and \eqref{eq:jac} is detailed in \cite{Muller1999d}.

Before proceeding further, we note that the incompressibility constraint \eqref{eq:Euler2} will often be codified in the appropriate function spaces through the subscript $\sigma$; thus, for instance, we write $L^2_\sigma \equiv \{v \in L^2 : \tp{div} \, v = 0\}$.

Solutions of \eqref{eq:Euler1}--\eqref{eq:Euler3} which fail to conserve energy have been studied extensively in relation to the so-called \emph{Onsager conjecture}, see \cite{Buckmaster2019,DeLellis2009,Isett2018} and the references therein. By a theorem of \textsc{Sz\'{e}kelyhidi} and \textsc{Wiedemann} \cite{Szekelyhidi2012}, for a dense set of initial data in $L^2_\sigma$ there exist infinitely many \emph{admissible} solutions $u \in L^\infty_t L^2_{\sigma,x}$ of \eqref{eq:Euler1}--\eqref{eq:Euler3}, that is, weak solutions which satisfy the energy inequality 
$$\int_{\R^n} |u(x,t)|^2 \d x \le \int_{\R^n} |u^0(x)|^2 \d x\qquad \tp{ for all }t \ge 0.$$ 
Moreover, the data $u^0$ can be chosen to be $C^\beta$ regular for any given $\beta \in (0,1/3)$, at least on the torus~\cite{Daneri2020}. 

For some data $u^0 \in L^2_\sigma$ there are even admissible compactly supported solutions. \textsc{Scheffer} had already constructed in \cite{Scheffer1993} solutions of the Euler equations which are compactly supported and square integrable in space-time, and a systematic study via convex integration was initiated by \textsc{De Lellis} and \textsc{Sz\'{e}kelyhidi} in the groundbreaking works~\cite{DeLellis2009,DeLellis2010}. Nevertheless, Theorem \ref{bigthm:incompatiblescalings} easily implies that, for a Baire-generic datum $u^0 \in L^2_\sigma$, the kinetic energy $\frac 1 2 \int_{\R^n} |u(x,t)|^2 \d x$ of weak solutions cannot undergo an $L^q$-type decay, $q < \infty$:

\begin{bigcor} \label{thm:Eulertheorem}
Take $n \ge 2$ and $2 < p < \infty$.
For every $M > 0$, the set of initial data for which there is a solution $u \in M \mathbb{B}_{L^p_t L^2_{x}}$ of the Cauchy problem \eqref{eq:Euler1}--\eqref{eq:Euler3} is nowhere dense in $L^2_\sigma$.
In particular, for a residual $G_\delta$ set of initial data in $L^2_\sigma$, the Cauchy problem \eqref{eq:Euler1}--\eqref{eq:Euler3} has no solution in $L^\infty_t L^2_x \cap [\bigcup_{2<p<\infty} L^p_t L^2_{x}]$.
\end{bigcor}

\begin{bigcor} \label{bigcor:Eulercorollary}
Take $\tau > 0$. An admissible solution $u \in L^\infty_t L^2_{\sigma,x}$ with $\tp{supp}(u) \subset \R^n \times [0,\tau]$ exists only for a nowhere dense set of data $u^0 \in L^2_\sigma$.
\end{bigcor}

To deduce Corollary \ref{thm:Eulertheorem} from Theorem \ref{bigthm:incompatiblescalings} we consider a \textit{linear} relaxation of the equations \eqref{eq:Euler1}--\eqref{eq:Euler3}; this is an idea in the spirit of \textsc{Tartar}'s framework for studying oscillations and concentrations in conservation laws \cite{Tartar1979, Tartar1983}. Such a relaxation is used here in order to render the associated solution-to-datum operator weak$^*$-to-weak$^*$ continuous. Corollary \ref{thm:Eulertheorem} is proved in \textsection \ref{sec:The incompressible Euler equations}.

The proof of Corollary \ref{thm:Eulertheorem} also applies to many other models in fluid dynamics. For instance, concerning the Navier--Stokes equations, we prove in an elementary fashion upper bounds for the generic energy dissipation rate of weak solutions. Another example is given by the equations of ideal magnetohydrodynamics, for which the analogue of Corollary \ref{thm:Eulertheorem} holds true. In that context bounded solutions with compact support in space-time were constructed in \cite{Faraco2020b}. On the torus $\T^3$, solutions in $L^\infty_t H^\beta_x$, for a small $\beta > 0$, violating magnetic helicity conservation were constructed in \cite{Beekie2020}.

\section{A nonlinear open mapping principle for positively homogeneous operators}
\label{sec:openmaps}

The main goal of this section is to prove Theorem \ref{thm:abstractopenmappingprinciple}. A related nonlinear uniform boundedness principle is proved in Proposition \ref{prop:PUB} and a precise statement concerning atomic decompositions in terms of $T$ is proved in Proposition \ref{prop:trichotomy}.

In the case of the Jacobian, by adapting a standard proof of the standard Open Mapping Theorem to Question \ref{qu:Lp} one obtains the following statement: if $\J(\dot{W}^{1,np}(\R^n,\R^n)) = \mathscr{H}^p(\R^n)$, then for every $f \in \mathscr{H}^p(\R^n)$ there exist $u,v \in \dot{W}^{1,np}(\R^n,\R^n)$ with 
\begin{equation} \label{eq:twoJacobians}
\tp{J} u + \tp{J} v = f\qquad \tp{ and } \qquad \int_{\R^n} (|\D u|^{np} + |\D v|^{np}) \d x \le C \Vert f \Vert_{\mathscr{H}^p}^p.
\end{equation}
Thus, quantitative control is gained at the expense of introducing an extra term $\tp{J} v$.

One could attempt to show the non-surjectivity of $\tp{J}$ by disproving the \emph{a priori} estimate in \eqref{eq:twoJacobians}. However, the extra term $\tp{J} v$ makes this a formidable task since the equation $\tp{J} u + \tp{J} v = f$ admits much more pathological solutions than $\tp{J} u = f$. As a prototypical example, there exist Lipschitz maps $u,v \colon \R^2 \to \R^2$ vanishing in the lower half-plane and satisfying $\tp{J} u + \tp{J} v = 1$ in the upper half-plane~\cite[Lemma 5]{Iwaniec2002}. In Theorem \ref{thm:abstractopenmappingprinciple} and Corollary \ref{cor:openmapJacobian}, the extra Jacobian $\tp{J} v$ is removed, leading to a genuinely nonlinear version of the Open Mapping Theorem.

\subsection{The proof of Theorem \ref{thm:abstractopenmappingprinciple}}
Here we give a slightly more precise version of Theorem  \ref{thm:abstractopenmappingprinciple}:

\begin{thm}
Let $X$ and  $Y$ be Banach spaces such that $\mb B_{X^*}$ is sequentially weak$^*$ compact. We make the following assumptions:

\begin{enumerate}[label={\normalfont (A\arabic*)}]
\item \label{it:weakcontinuity} $T\colon X^*\to Y^*$ is a weak$^*$-to-weak$^*$ sequentially continuous operator.

\item \label{it:homogeneity} $T(a u) = a^s T(u)$ for all $a > 0$ and $u \in X^*$, where $s > 0$.

\item \label{it:isometries} For $k\in \N$ there are isometric isomorphisms $\sigma_k^{X^*}\colon X^*\to X^*$, $\sigma_k^{Y^*}\colon Y^*\to Y^*$ such that 
\begin{align*}
 T \circ \sigma_k^{X^*} = \sigma_k^{Y^*} \circ T \qquad \text{for all } k \in \N, \qquad
\sigma_k^{Y^*} f \overset{*}{\rightharpoonup} 0 \qquad \text{for all } f \in Y^*.
\end{align*}
\end{enumerate}

\noindent Then the following conditions are equivalent:
\begin{enumerate}
\item\label{it:nonmeagre} $T(X^*)$ is non-meagre in $Y^*$.

\item\label{it:surjectivity} $T(X^*) = Y^*$.

\item\label{it:openat0} $T$ is open at the origin.

\item\label{it:quantitativesolvability} For every $f \in Y^*$ there exists $u \in X^*$ such that
\begin{equation} \label{eq:quantitative solvability}
T u = f, \qquad \Vert u \Vert_{X^*}^s \le C \Vert f \Vert_{Y^*}.\end{equation}
\end{enumerate}
\end{thm}

A sufficient condition for $\mb B_{X^*}$ to be sequentially weak$^*$ compact is that $X$ is a \textit{weak Asplund space} \cite[Theorem 3.5]{Stegall1981}. For instance, reflexive or separable spaces are weak Asplund \cite{Deville1993}.

\begin{proof}[Proof of Theorem \ref{thm:abstractopenmappingprinciple}]
We have \ref{it:quantitativesolvability} $\Rightarrow$ \ref{it:openat0} $\Rightarrow$ \ref{it:surjectivity} $\Rightarrow$ \ref{it:nonmeagre} and so we just prove \ref{it:nonmeagre} $\Rightarrow$ \ref{it:quantitativesolvability}.

Assume that \ref{it:nonmeagre} holds. We may write $T(X^*)$ as a union $\cup_{\ell=1}^\infty K_\ell$, where
\[K_\ell \equiv \big\{f \in Y^* : \text{there exists } u \in X^* \text{ with } T u = f \text{ and } \Vert u \Vert_{X^*}^s \le \ell \Vert f \Vert_{Y^*}\big\}.\]
Since balls in $X^*$ are sequentially weak$^*$ compact, by \ref{it:weakcontinuity}, the sets $K_\ell$ are norm-closed. Now, by the Baire Category Theorem, some $K_\ell$ contains a closed ball $\bar{B}_r(f_0)$.

Our aim is to solve \eqref{eq:quantitative solvability} whenever $\Vert f \Vert_{Y^*} = r$; assumption \ref{it:homogeneity} then implies the claim. Suppose, therefore, that $\Vert f \Vert_{Y^*} = r$. For every $k \in \N$ we have $f_0 + (\sigma_k^{Y^*})^{-1} f \in \bar{B}_r(f_0)$. Hence, we may choose $u_k \in X^*$ such that $T u_k = f_0 + (\sigma_k^{Y^*})^{-1} f$ and
\[\Vert \sigma^{X^*}_k u_k \Vert_{X^*}^s =  \Vert u_k \Vert_{X^*}^s \le \ell \Vert f_0 + (\sigma_k^{Y^*})^{-1} f \Vert_{Y^*} \le \ell (\Vert f_0 \Vert_{Y^*} + r).\]
Since balls in $X^*$ are sequentially weak$^*$ compact, after passing to a subsequence if need be, $\sigma_k^{X^*} u_k$ converges weakly$^*$ to some $u \in X^*$, so that $T(\sigma_k^{X^*} u_k) \overset{*}{\rightharpoonup} T u$. By the lower semicontinuity of the norm we have $$\Vert u \Vert_{X^*}^s \le \liminf_{k \to \infty} \Vert \sigma_k^{X^*} u_k \Vert_{X^*}^s \le \ell (\Vert f_0 \Vert_{Y^*} + r).$$
On the other hand, \ref{it:isometries} gives
\[T (\sigma_k^{X^*} u_k) = \sigma_k^{Y^*} (T u_k) = \sigma_k^{Y^*} f_0 + f \overset{*}{\rightharpoonup} f,\]
so that, by \ref{it:weakcontinuity}, $T u = f$. Thus $u$ solves \eqref{eq:quantitative solvability} and the proof is complete.
\end{proof}

The theory of Compensated Compactness provides many examples of nonlinear operators to which Theorem \ref{thm:abstractopenmappingprinciple} applies. Here we give a general formulation in the spirit of \cite{Guerra2019}, see also \cite{Murat1981,Tartar1979}, which we then illustrate with more concrete examples.

\begin{ex}
Let $\mc A$ be an $l$-th order homogeneous linear operator, which for simplicity we assume to have constant coefficients; that is, for $v\in C^\infty(\R^n,\mb V)$,
$$\mc A v=\sum_{|\a|=l} A_\a \p^\a v, \qquad A_\a \in \tp{Lin}(\mb V,\mb W),$$
where $\mb V, \mb W$ are finite-dimensional vector spaces. For $p\in [1,+\infty)$ and $s\in \N$, $s\geq 2$, take 
$$X^*=L^{ps}_{\mc A}(\R^n,\mb V), \qquad Y^* = \mathscr H^p(\R^n).$$ Here $L^{ps}_{\mc A}(\R^n,\mb V)$ is the space of those $v\in L^{ps}(\R^n,\mb V)$ such that $\mc A v=0$ in the sense of distributions.
We will further need the following standard non-degeneracy assumption:
\begin{equation}
\label{eq:constantrank}
\tp{the symbol of } \mc A, \tp{ seen as a matrix-valued polynomial, has constant rank.}
\end{equation}
Whenever \eqref{eq:constantrank} holds, we say that $\mc A$ has \textit{constant rank}. We will not discuss this assumption here but it holds in all of the examples below; the reader may find other characterizations of constant rank operators in \cite{GuerraRaita2020, Raita2018}.

Let $T\colon X^*\to Y^*$ be a homogeneous sequentially weakly continuous operator. Under the assumption \eqref{eq:constantrank}, such operators were completely characterised in \cite{Guerra2019}, and they are often called \textit{Compensated Compactness quantities}. They can be realised as certain constant-coefficient partial differential operators and so they necessarily satisfy \ref{it:isometries} if one takes the isometries $\sigma_k^{X^*},\sigma_k^{Y^*}$ to be translations. The following are standard examples of such operators:
\begin{enumerate}
\item $\mc A=\tp{curl}$ and $T=\J$. For this example, take $\mb V=\R^{n\times n}$ and choose $\mc A$ in such a way that $\mc A v =0$ if and only if $v=\D u$, for some $u\colon \R^n\to \R^n$. For instance, we may take $(\tp{curl}\, v)_{ijk}=\p_k v_{ij}-\p_j v_{ik}$. We also choose $s=n$ and so $X^*={\dot W}^{1,np}(\R^n,\R^n)$. The only positively $n$-homogeneous sequentially weakly continuous operator $X^*\to Y^*$ is the Jacobian, and in particular we recover Corollary \ref{cor:openmapJacobian}.
\item $\mc A=\tp{curl}^2$ and $T=\tp{H}$. Here $\mc A$ is chosen similarly to the previous example, but now $\mc A v=0$ if and only if $v=\D^2 u$, for some $u\colon \R^n\to \R$. Again we take $s=n$ and so $X^*={\dot W}^{2,np}(\R^n,\R^n)$. We may take $T=\tp{H}\colon X^*\to Y^*$ to be the Hessian, and Theorem \ref{thm:abstractopenmappingprinciple} shows that it satisfies the open mapping principle.
\end{enumerate}
The two previous examples admit a straightforward generalisation, where one considers $s$-th order minors (instead of the determinant) and a $j$-th order curl (instead of $j=1,2$).
\begin{enumerate}[resume]
\item $\mc A=(\tp{div},\,\tp{curl})$ and $T=\langle\cdot, \cdot\rangle$. In this example, $s=2$ and $T$ is the standard inner product acting on a pair $v\equiv (B,E)\colon \R^n\to \R^n\times \R^n$; here, $B$ is thought of as a ``magnetic field'' and $E$ as an ``electric field''. As before, Theorem \ref{thm:abstractopenmappingprinciple} shows that $T$ satisfies the open mapping principle.
\end{enumerate}
\end{ex}

We conclude this subsection by comparing the above example with \cite{Coifman1993}. There, the authors address the problem of deciding whether Compensated Compactness quantities are surjective, particularly when $p=1$. Thus Theorem \ref{thm:abstractopenmappingprinciple} can be read as saying that openness at zero is a necessary condition for a positive answer to this problem.

\subsection{A nonlinear uniform boundedness principle}

We also present a nonlinear version of the Uniform Boundedness Principle in the spirit of Theorem \ref{thm:abstractopenmappingprinciple}; under certain structural conditions, a family of operators which is pointwise bounded in a neighbourhood of the origin is uniformly bounded in a (possibly smaller) neighbourhood of the origin.

\begin{prop} \label{prop:PUB}
Let $X$ and $Z$ be Banach spaces and let $I$ be an index set. Suppose the following conditions hold:

\begin{enumerate}
\item\label{it:wslsc} For every $i \in I$, the mapping $T_i \colon X \to Z$ is such that $u \mapsto \Vert T_i u \Vert_Z \colon X \to \R$ is weakly sequentially lower semicontinuous.

\item\label{it:pointwisebdd} There is $\e>0$ such that $\sup_{i \in I} \Vert T_i(u) \Vert_Z < \infty$ whenever $\Vert u \Vert_X \le \e$.

\item\label{it:isometriesuniformboundedness} For $j\in \N$ there are isometric isomorphisms $\sigma_k^X\colon X\to X$ and $\sigma_k^Z\colon Z\to Z$ such that
\begin{align*}
& T_i \circ \sigma_k^X = \sigma_k^Z \circ T_i \qquad \text{for all } i \in I \text{ and } k \in \N, \\
&\sigma_k^X u \rightharpoonup 0 \qquad \text{for all } u \in X.
\end{align*}
\end{enumerate}
Then there exists $\delta > 0$ such that
\[\sup_{\Vert u \Vert_X \le \delta} \sup_{i \in I} \Vert T_i u \Vert_Z < \infty.\]
\end{prop}

\begin{proof}
By \ref{it:pointwisebdd}, we may write $\e \mathbb{B}_X = \cup_{\ell=1}^\infty C_\ell$, where $C_\ell \equiv \{u \in \e \mathbb{B}_X \colon \sup_{i \in I} \Vert T_i u \Vert_Z \le \ell\}$ and \ref{it:wslsc} shows that each $C_\ell$ is norm closed. Thus, by the Baire Category Theorem, some $C_\ell$ contains a closed ball $\bar{B}_\delta(u_0)$.

Let now $\Vert u \Vert_X \le \delta$ and $i \in I$. By \ref{it:isometriesuniformboundedness}, we have  $u + \sigma_k^X u_0= \sigma_k^X [u_0 + (\sigma_k^X)^{-1} u]  \in \bar{B}(u_0,\delta)$ and moreover $u + \sigma_k^X u_0\rightharpoonup u$. So by \ref{it:wslsc} and again \ref{it:isometriesuniformboundedness}, we have
\begin{align*}
\Vert T_i u \Vert_Z
&\le \liminf_{k \to \infty} \Vert T_i \sigma_k^X [u_0 + (\sigma_k^X)^{-1} u] \Vert_Z
= \liminf_{k \to \infty} \Vert \sigma_k^Z T [u_0 + (\sigma_k^X)^{-1} u] \Vert_Z \le \ell.
\end{align*}
The proof is complete.
\end{proof}

We note that, in the linear case, it is possible to prove the Banach--Steinhaus Uniform Boundedness Principle without using Baire's Category Theorem: the proof relies, instead, on the so-called ``gliding hump method''. For an extension of the classical Uniform Boundedness Principle using this method, we refer the reader to  \cite{Gal1951}.

\subsection{Atomic decompositions in terms of \texorpdfstring{$T$}{}}
The main motivation behind this subsection is Theorem \ref{thm:CLMS/H}. It establishes an analogue of the atomic decomposition of $\mathscr{H}^1(\R^n)$, giving a weak factorization on $\mathscr{H}^p(\R^n)$ in the spirit of the classical work of \textsc{Coifman}, \textsc{Rochberg} and \textsc{Weiss} \cite{Coifman1976}:

\begin{thm}\label{thm:CLMS/H}
Let $p\in [1,\infty)$. 
For every $f\in \mathscr{H}^p(\R^n)$ there are functions $u_i \in \dot{W}^{1,np}(\R^n,\R^n)$ and real numbers $c_i$ such that
\begin{equation} \label{eq:WeakFactorization}
f=\sum_{i=1}^\infty c_i \J u_i, \hspace{1cm}  \Vert u_i \Vert_{\dot{W}^{1,np}(\R^n)} \leq 1, \hs \sum_{i=1}^\infty |c_i| \lesssim \Vert f \Vert_{\mathscr{H}^p(\R^n)}.
\end{equation}
In particular, $\mathscr{H}^p(\R^n)$ is the smallest Banach space containing the range $\J({\dot W}^{1,np}(\R^n,\R^n))$.
\end{thm}

Theorem \ref{thm:CLMS/H} was proved in \cite{Coifman1993} for $p=1$, while the case $p>1$ is much harder and was established only recently by \textsc{Hyt\" onen} in \cite{Hytonen2018}. 
It is conceivable that the operator $\tp{J} \colon \dot{W}^{1,np}(\R^n,\R^n) \to \mathscr{H}^p(\R^n)$ is not surjective but \eqref{eq:WeakFactorization} improves to a \emph{finitary} decomposition of $\mathscr{H}^p(\R^n)$ in terms of Jacobians. In Proposition \ref{prop:trichotomy}, we formulate a rather precise classification of infinitary and finitary decompositions in the setting of Theorem \ref{thm:abstractopenmappingprinciple}.

Take $\omega \in \overline \N \equiv \N\cup \{\infty\}$. Given $T$ as in Theorem \ref{thm:abstractopenmappingprinciple}, if every $f \in Y^*$ can be written as
\begin{equation} \label{eq:1/n-surjectivity}
f = \sum_{j=1}^\omega c_j\, T u_j, \qquad c_j \in \R, \, u_j \in \mathbb{B}_{X^*},
\end{equation}
then, following \cite{Dixon1988}, $T$ is said to be \emph{$1/\omega$-surjective}. If, furthermore,
\begin{equation} \label{eq:1/n-openness}
\sum_{j=1}^\omega |c_j| \lesssim \Vert f \Vert_{Y^*}
\end{equation}
for all $f \in Y^*$, then $T$ is said to be \emph{$1/\omega$-open}. \textsc{Dixon} \cite{Dixon1988} generalised \textsc{Horowitz}'s example by constructing, for every $m \in \N$, a continuous $1/m$-surjective bilinear map between Banach spaces which is not $1/m$-open. In fact, in Dixon's notation, the constants $c_j$ are subsumbed by the elements $u_j$. The formalism \eqref{eq:1/n-surjectivity}--\eqref{eq:1/n-openness} is, however, more standard in the context of atomic decompositions.

In Proposition \ref{prop:trichotomy} we show that, for $\omega\in \overline \N$, and under the assumptions of Theorem \ref{thm:abstractopenmappingprinciple}, $1/\omega$-surjectivity implies $1/\omega$-openness. 

\begin{prop} \label{prop:trichotomy}
Suppose $X$, $Y$ and $T$ satisfy the assumptions of Theorem \ref{thm:abstractopenmappingprinciple}. 
Let us define, for $\omega \in \overline \N$, the sets $$\Lambda_\omega \equiv \biggr\{\sum_{j=1}^\omega c_j T u_j : u_j\in \mb B_{X^*},\, c_j\in \R  \tp{ and } \sum_{j=1}^\omega |c_j|<\infty\biggr\}.$$
If $\Lambda_\infty$ is not meagre in $Y^*$, there is $\omega\in \overline \N$ such that $\Lambda_\omega=Y^*$ and $\bigcup_{m<\omega} \Lambda_{m}$ is meagre in $Y^*$; moreover, $T$ is $1/\omega$-open.
\end{prop}

\begin{proof}
We show that if $\bigcup_{m<\infty} \Lambda_m$ is not meagre in $Y^*$, then there is $m\in \N$ such that $\Lambda_m=Y^*$ and $\Lambda_{m-1}$ is meagre in $Y^*$. Note that, for each $m\in \N$, the set $\Lambda_m$ is closed;
it follows from the Baire Category Theorem that one of the sets $\Lambda_m$ contains a ball. By using the $s$-homogeneity of $T$, we write $\Lambda_m = \{\sum_{j=1}^m d_j \, T v_j \colon d_j \in \R, \, v_j \in X^*\}$. By applying Theorem \ref{thm:abstractopenmappingprinciple} to the $(s+1)$-homogeneous operator
 $$\tilde{T} \colon \R^m \times (X^*)^m \to Y^*, \qquad \tilde{T}\big(\{d_j\}_{j=1}^m,\{v_j\}_{j=1}^m\big) \equiv \sum_{j=1}^m d_j T v_j\,,$$
we find that for each $f\in Y^*$ there are $d_j \in \R$ and $v_j \in X^*$ such that 
\begin{equation} \label{eq:m-ary atomic decomposition}
\sum_{j=1}^m d_j T v_j = f, \qquad
\sum_{j=1}^m (|d_j|^{s+1} + \Vert v_j\Vert_{X^*}^{s+1}) \lesssim \Vert f \Vert_{Y^*}.
\end{equation}
We set $c_j = d_j \Vert v_j \Vert_{X^*}^s$ and denote $u_j = v_j/\Vert v_j \Vert_{X^*}$ if $v_j \neq 0$ and $u_j = 0$ if $v_j = 0$. Thus $c_j\, T u_j = d_j \,T v_j$ for $j=1,\ldots,m$. Consequently, through Young's inequality, \eqref{eq:m-ary atomic decomposition} yields
\begin{equation} \label{eq:m-ary atomic decomposition 2}
\sum_{j=1}^m c_j\, T u_j = f, \qquad \sum_{j=1}^m |c_j| \lesssim \Vert f \Vert_{Y^*}, \qquad u_j \in \mathbb{B}_{X^*}.
\end{equation}
It now suffices choose the smallest $m\in \N$ such that $T\colon X^*\to Y^*$ is $1/m$-surjective; the $1/m$-openness of $T$ is given by \eqref{eq:m-ary atomic decomposition 2}.

We finally show that if $\bigcup_{m<\infty} \Lambda_m$ is meagre but $\Lambda_\infty$ is non-meagre, then in fact $\Lambda_\infty = Y^*$ and 
$T$ is $1/\infty$-open. We denote $V \equiv \{\e \,T u \colon \e = \pm 1, \, u \in \mathbb{B}_{X^*}\} \subset Y^*$. Now $V$ is bounded and symmetric and, by assumption, $\{\sum_{j=1}^\infty c_j f_j : f_j \in V \tp{ for all }j \tp{ and } \sum_{j=1}^\infty |c_j| < \infty\}$ is non-meagre in $Y^*$. By ~\cite[Lemma 3.1]{Lindberg2017}, $\{\sum_{j=1}^\infty c_j T u_j \colon \sum_{j=1}^\infty |c_j| = 1, \, u_j \in \mathbb{B}_{X^*}\} \subset Y^*$ contains a ball centred at the origin. It immediately follows that given $f \in Y^*$, conditions \eqref{eq:1/n-surjectivity}--\eqref{eq:1/n-openness} can be satisfied with $\omega = \infty$.
\end{proof}

\begin{remark}
It is tempting to try and prove the last part of Proposition \ref{prop:trichotomy} by defining an auxiliary operator $\tilde{T} \colon \ell^{s+1}(\N) \times \ell^{s+1}(\N;X^*) \to Y^*$ via $T(\{d_j\}_{j=1}^\infty, \{v_j\}_{j=1}^\infty) \equiv \sum_{j=1}^\infty d_j T v_j$ and using Theorem \ref{thm:abstractopenmappingprinciple} on $\tilde{T}$, in analogy to the case $\omega < \infty$. However, such an operator is never weak$^*$-to-weak$^*$ continuous unless $T \equiv 0$. Indeed, suppose $T u \neq 0$ and set $d_{jk} = \delta_{jk}$ and $v_{jk} = \delta_{jk} u$. Now $\tilde{T}(\{d_{jk}\}_{j=1}^\infty, \{v_{jk}\}_{j=1}^\infty) = T u$ for all $k \in \N$ but $(\{d_{jk}\}_{j=1}^\infty, \{v_{jk}\}_{j=1}^\infty) \overset{*}{\rightharpoonup} 0$.
\end{remark}

\begin{ex}
Let us denote by $\mc H$ the Hilbert transform and by $T \colon L^2(\R,\R) \to \mathscr{H}^1(\R)$ the operator $T(\chi, \eta) \equiv \mathcal{H} \chi \,\mathcal{H} \eta - \chi\, \eta $. The strong factorization $\mathscr{H}^1(\C_+) = \mathscr{H}^2(\C_+) \cdot \mathscr{H}^2(\C_+)$ of analytical Hardy spaces, see e.g.\ \cite{Lindberg2020} for a proof, yields the surjectivity result
\begin{equation} \label{eq:surjectivityonR}
\mathscr{H}^1(\R) = \left\{T(\chi,\eta): \chi, \eta \in L^2(\R)\right\}.
\end{equation}
Thus, in this case, $\Lambda_1=\mathscr H^1(\R)$. 

Another example is obtained by considering the operator $\J\colon W^{1,np}(\R^n,\R^n)\to \mathscr{H}^p(\R^n)$, where $n \ge 2$ and $p \in [1,\infty)$; we emphasise that the Sobolev space is \emph{inhomogeneous}. In this case, $\Lambda_\infty$ is meagre in $\mathscr H^p(\R^n)$, see \cite{Lindberg2015} and Corollary \ref{cor:jacobianinhomogeneous}. However, if we instead consider the Jacobian as defined on ${\dot W}^{1,np}$, then $\Lambda_\infty=\mathscr H^p(\R^n)$ by the results of \cite{Hytonen2018}, although it is unclear whether this is optimal. 
We note that for $\J\colon {\dot W}^{1,2p}(\R^2,\R^2)\to\mathscr H^p(\R^2)$, the statement $\Lambda_1=\mathscr H^p(\R^2)$ is equivalent to 
$$\mathscr{H}^p(\R^2) = \left\{|\mathcal{S} \omega|^2 - |\omega|^2 \colon \omega \in L^{2p}(\R^2,\R^2)\right\},$$ 
compare with \eqref{eq:surjectivityonR}. Here $\mathcal{S}$ is the Beurling--Ahlfors transform, which one may think of as the square of a complex Hilbert transform \cite{Iwaniec2001}.
\end{ex}

We are not aware of operators satisfying the assumptions of Theorem \ref{thm:abstractopenmappingprinciple} and for which there is $1<m\in \N$ such that $\Lambda_m=Y^*$ but $\bigcup_{m'<m}\Lambda_{m'}$ is meagre in $Y^*$.

\section{Tools from Geometric Function Theory}
\label{sec:prelims}

This section collects, for the convenience of the reader, useful known results about Sobolev maps and mappings of finite distortion. These results will only be used in relation to the Jacobian determinant in Section \ref{sec:applicationstoJacobian}.

\subsection{The Lusin (N) property and the change of variables formula}

The following notions are very relevant in relation to the change of variables formula:

\begin{ndef}
  Let $u\colon \Omega\to \R^n$ be a continuous map which is
  differentiable a.e.\ in $\Omega$. Then:
  \begin{enumerate}
  \item $u$ has the \textit{Lusin (N) property} if
    $|u(E)|=0$ for any $E\subset \Omega$ such that $|E|=0$;
  \item $u$ has the \textit{(SA) property} if $|u(E)|=0$ for
    any open set $E\subset\Omega$ with $\J u=0$ a.e.\ in $E$.
  \end{enumerate}
\end{ndef}

In the one-dimensional case, the Lusin (N) property is well understood: for instance, on an interval, a continuous function of bounded variation has the Lusin (N) property if and only if it is absolutely continuous. However, in higher dimensions, the situation is much more complicated, although we have the following characterisation, proved in \cite{Martio1992}:

\begin{prop}\label{prop:SAproperty}
Let $u\in W^{1,n}(\Omega,\R^n)$ be a continuous map with $\J u\geq 0$ in $\Omega$. Then $u$ has the Lusin (N) property if and only if it has the (SA) property.
\end{prop}

We remark that Proposition \ref{prop:SAproperty} is in general false if $\J u\not \geq 0$, see \cite{Reshetnyak1987} for a counterexample. The following result, see \cite{Maly1995}, is also useful for our purposes:

\begin{prop}\label{prop:pseudomonotonemappingsLusin}
Let $u\in W^{1,n}(\Omega,\R^n)$ be a continuous map such that, for some $K\geq 1$,
\begin{equation}
\tp{diam}(u(B_r(x))\leq K \tp{diam}(u(\p B_r(x))\qquad \tp{ for all } B_r(x)\Subset \Omega.
\label{eq:pseudomonotone}
\end{equation}
Then $u$ has the Lusin (N) property.
\end{prop}

The change of variables formula is closely related to the Jacobian determinant:

\begin{thm}\label{thm:areaformula}
Let $u\in C^0(\Omega,\R^n)\cap W^{1,n}(\Omega,\R^n)$ be a map with the Lusin (N) property. Then
\begin{equation}
\int_E |\J u| \d x = \int_{\R^n} \mc N(y,u,E) \d y \qquad \tp{ for all measurable sets } E\subset \Omega,
\label{eq:areaformula}
\end{equation}
where $\mc N$ is the {\normalfont multiplicity function}, defined as $\mc N(y,u,E)\equiv \#\{x\in E: u(x)=y\}$. 
\end{thm}

The reader may find the proof of Theorem \ref{thm:areaformula}, together with a wealth of information on geometric properties of Sobolev maps, in \cite{Fonseca1995}.

\subsection{Mappings of finite distortion}

In this subsection we recall some useful facts about mappings of finite distortion and, for simplicity, we focus on the planar case $n=2$, see \cite{Astala2009}. The reader can also find these and higher-dimensional results in \cite{Hencl2014a, Iwaniec2001}.

\begin{ndef}\label{def:FD}
Let $u\in W^{1,1}_\tp{loc}(\Omega,\R^2)$ be such that $0 \leq \J u\in L^1_\tp{loc}(\Omega)$. We say that $u$ is a \emph{map of finite distortion} if there is a function $K\colon \Omega\to [1,\infty]$  such that  $K<\infty$ a.e.\ in $\Omega$ and
$$|\D u(x)|^2 \leq K(x)\, \J u(x) \quad \tp{ for a.e.\ } x \tp{ in } \Omega.$$
If $u$ has finite distortion, we can set $\tp{K}u(x) = \frac{|\D u|^2}{\J u(x)}$ if $\J u(x)\neq 0$ and $\tp{K}u(x)=1$ otherwise; this function is the (optimal) \emph{distortion} of $u$.
\end{ndef}

We note that, in Definition \ref{def:FD}, $|\cdot |$ denotes the operator norm of a matrix.

We summarise here some of the key analytic and topological properties of mappings of finite distortion in the plane:

\begin{thm}\label{thm:theoryofFDmaps}
Let $\Omega\subset \R^2$ and let $u\in W_\tp{loc}^{1,2}( \Omega, \R^2)$ be a  map of finite
distortion. Then:
\begin{enumerate}
\item $u$ has a continuous representative and, whenever $r<R$ and $B_R(x_0)\subset \Omega$, $$\left(\tp{osc}_{B_r(x_0)} u\right)^2\leq \frac{C}{\log(R/r)} \int_{B_R(x_0)} |\D u|^2 \d x;$$
\item $u$ has the Lusin (N) property;
\item $u$ is differentiable a.e.\ in $\Omega$;
\item \label{it:open}  if $\tp{K} u\in L^1(\Omega)$ then $u$ is open and discrete;
\item \label{it:multiplicity} if $\tp{K} u \in L^1(\Omega)$ then for each $\Omega'\Subset \Omega$ there is $m=m(\Omega')$ such that 
$$\mc N(y,u,\Omega')\leq m \qquad \tp{ for all }y\in u(\Omega').$$
\end{enumerate}
\end{thm}

Whenever $u$ is a map of finite distortion we always implicitly assume that $u$ denotes the continuous representative of the equivalence class in $W^{1,2}_\tp{loc}(\Omega,\R^2)$. If $u$ is such that $\tp{K} u\in L^1(\Omega)$, we say that $u$ has \textit{integrable distortion}; the theory of such maps was pioneered in \cite{Iwaniec1993a}.

We remark that the first three properties of Theorem \ref{thm:theoryofFDmaps} are a consequence of the fact that mappings of finite distortion are \textit{monotone in the sense of Lebesgue}:

\begin{prop}\label{prop:FDimpliesmonotone}
Let $u\in W^{1,2}_\tp{loc}(\Omega,\R^2)$ be a map of finite distortion; then  \eqref{eq:pseudomonotone} holds. In fact, if we measure the diameter in $\R^2$ with respect to the $\ell^\infty$ norm, we can take $K=1$.
\end{prop}

\section{Applications to the Jacobian equation}
\label{sec:applicationstoJacobian}

This section expands on the relation between Theorem \ref{thm:abstractopenmappingprinciple} and Question \ref{qu:Lp}. 
We begin by discussing submersions, as it is well-known that they  form a subclass of the class of open operators. We show that, for $p\in [1,2)$, there are no submersions between the spaces in \eqref{eq:spaces}. As a consequence, we deduce in Corollary \ref{cor:JacobianNondifferentiable} the non-differentiability of an hypothetical right inverse of the Jacobian.
In a different direction, we also combine the results from Section \ref{sec:prelims} with Theorem \ref{thm:abstractopenmappingprinciple} in order to study solutions of the Jacobian equation, and in particular we prove Theorem \ref{bigthm:nicesols}.

\subsection{The Jacobian is a submersion nowhere}
\label{sec:submersion}

Let $X, Y$ be Banach spaces. We use the following terminology:

\begin{ndef}
An operator $T\colon X\to Y$ is said to be a \textit{submersion at} $x_0\in X$ if $T$ is G\^{a}teaux-differentiable at $x_0$ and $T'(x_0)\colon X\to Y$ is onto. It is said to be a \textit{regular submersion at} $x_0$ if additionally $\ker T'(x_0)$ is complemented in $X$.
\end{ndef}

We note that, in the literature, the word submersion often refers to a regular submersion. By analogy to the finite-dimensional case, if $T$ is a submersion at $x_0$ then it is open at $x_0$, see for instance \cite[Corollary 15.2]{Deimling1985}:

\begin{thm}\label{thm:submersion}
Let $T\colon X\to Y$ be a locally Lipschitz submersion at $x_0\in X$. For all $R>0$ sufficiently small, there is $r>0$ such that $B_r(T(x_0))\subseteq T(B_R(x_0))$.
\end{thm}

The submersion condition also plays an important role in Lyusternik's theory of constrained variational problems, through the existence of Lagrange multipliers. 
We remark that, in that setting, it is customary to deal with regular submersions. Here we do not discuss further the existence of Lagrange multipliers nor their properties, referring instead the interested reader to \cite[§43]{Zeidler1985} for their general theory. In the context of Question \ref{qu:Lp}, Lagrange multipliers were considered in the third author's doctoral thesis \cite{Lindberg2015}.

The main result of this subsection is Proposition \ref{prop:Jacobiannotsubmersion}, which shows that Theorem \ref{thm:submersion} does not apply to the Jacobian. 
We begin with the following straightforward lemma:

\begin{lemma}
Suppose $T\colon X\to Y$ is G\^{a}teaux-differentiable. If $Y^*$ does not embed into $X^*$ then $T$ is a submersion at no point.
\end{lemma}

\begin{proof}
We prove the contrapositive. Suppose $T$ is a submersion at some $x_0\in X$, that is, $L\equiv T'(x_0)\colon X\to Y$ is onto. By the classical open mapping theorem, $L^*\colon Y^*\to X^*$ is bounded from below and is thus an isomorphism onto its image. Thus $Y^*$ embeds into $X^*$.
\end{proof}

The main result of this section is the following:

\begin{prop}\label{prop:Jacobiannotsubmersion}
Let $p\in [1,2)$ and suppose $T \colon {\dot W}^{1,np}(\R^n,\R^n)\to \mathscr H^p(\R^n)$ is G\^{a}teaux-differentiable. Then $T$ is a submersion at no point.
\end{prop}

The range of $p$ in Proposition \ref{prop:Jacobiannotsubmersion} is optimal, see Remark \ref{remark:optimalrange}.

\begin{proof}
The case $p=1$ is simple: $(\mathscr H^1(\R^n))^*=\tp{BMO}(\R^n)$ is not reflexive and thus it cannot embed into a reflexive space, such as ${\dot W}^{1,n}(\R^n,\R^n)^*$.

For $p\in(1,2)$, we begin by using the isomorphism $(-\Delta)^{1/2} \colon {\dot W}^{1,np}(\R^n,\R^n)\to L^{np}(\R^n,\R^n)$.
Thus it suffices to show that $L^{p'}$ does not embed into $L^{(np)'}$ for $p\in (1,2)$, where $q'$ denotes the H\"older conjugate of $q$. 
Since $p'>2$, we appeal to Lemma \ref{lemma:rangeofq} below to finish the proof.
\end{proof}

Thus, it remains to prove the next lemma, where $H$ is a Hilbert space.

\begin{lemma}\label{lemma:rangeofq}
Let $p\in [1,2], q\in [1,\infty)$. If $L^q(\R^n)$ embeds into $L^p(\R^n,H)$ then $1\leq p \leq q \leq 2$.
\end{lemma}

This result is well-known to the experts and a very complete statement can be found in \cite[Theorem 6.4.20 and Proposition 12.1.10]{Albiac2016}, which we quote here:

\begin{prop}\label{prop:subspacesofLp}
Let $p,q \in [1,\infty)$. Then $L^q(\R^n)$ embeds into $L^p(\R^n)$ if and only if one of the following conditions holds:
\begin{enumerate}
\item\label{it:p<2} $1 \le p \le q \le 2$,
\item\label{it:p>2} $2 < p < \infty$ and $q \in \{2,p\}$.
\end{enumerate}
Moreover, if $1<p,q$ then $L^q$ embeds complementably into $L^p$ if and only if $q\in\{2,p\}$.
\end{prop}

Lemma \ref{lemma:rangeofq} is essentially deduced from Proposition \ref{prop:subspacesofLp}, as the vector-valued $L^p$ space poses only minor changes to the proof. We sketch a proof of Lemma \ref{lemma:rangeofq} here, in order to improve the readability of the paper. The proof relies on the notions of (Rademacher) type and cotype of a Banach space:

\begin{ndef}
Let $(\e_i)_{i=1}^\infty$ be a sequence of i.i.d.\ random variables such that
$$\mb P(\e_i=1)=\mb P(\e_i=-1)=\frac 1 2.$$

A Banach space $X$ has \textit{type $p$}, $p\in [1,2]$ if there is a constant $C$ such that
$$
\biggr(\mb E\bigr \Vert\sum_{i=1}^n \e_i x_i\bigr \Vert^p\biggr)^{1/p}
\leq C\biggr(\sum_{i=1}^n \Vert x_i\Vert^p\biggr)^{1/p}
$$
for any vectors $x_i\in X$. Likewise, $X$ has \textit{cotype $q$}, $q\in [2,+\infty]$, if there is $C$ such that
$$
\biggr(\sum_{i=1}^n \Vert x_i\Vert^q\biggr)^{1/q} 
\leq C\biggr(\mb E\bigr \Vert\sum_{i=1}^n \e_i x_i\bigr \Vert^q\biggr)^{1/q}
$$
for any vectors $x_i\in X$.
\end{ndef}

The range of $p$ and $q$ in the definitions of type and cotype are natural and are determined by Khintchine's inequality. Moreover, if $X$ is of type $p$ then it is also of type $r$ for any $r<p$; if it is of cotype $q$, it is also of cotype $r$ for any $r>q$. 
\begin{ex}\label{ex:typeandcotype}
As before, $X$ is a Banach space.
\begin{enumerate}
\item A Hilbert space $H$ has type and cotype 2: this follows from the parallelogram law.
\item If $X$ has type $p$ then $X^*$ has cotype $p'$, although the converse is not true.
\item\label{it:optimality} If $p\in [1,2]$ then $\ell^p$ has type $p$ and if $p\in [2,+\infty]$ then $\ell^p$ has cotype $p$. Moreover, these values are optimal, as can be seen by considering the standard basis.
\item If $X$ has type $p$ and cotype $q$, the space $L^r(\R^n,X)$ has type $\min\{r,p\}$ and cotype $\max\{r,q\}$.
\end{enumerate}
\end{ex}

The reader may find details and further examples in \cite{Albiac2016, Hytonen2017}.

\begin{proof}[Proof of Lemma \ref{lemma:rangeofq}]
Clearly type and cotype are inherited by subspaces. Thus, if $p\in [1,2]$ and if $L^q(\R^n)$ embeds into $L^p(\R^n,H)$, then $L^q(\R^n)$ must have type $p$ and cotype 2. Since $\ell^q$ embeds into $L^q(\R^n)$, the same can be said for $\ell^q$. Hence, the optimality in Example \ref{ex:typeandcotype}\ref{it:optimality} shows that $p\leq q \leq 2$.
\end{proof}

\begin{remark}
Inspection of the proof reveals that, in Proposition \ref{prop:Jacobiannotsubmersion}, the following stronger conclusion holds: for any $u\in {\dot W}^{1,np}(\R^n,\R^n)$, $T'_u\colon {\dot W}^{1,np}(\R^n,\R^n)\to \mathscr H^p(\R^n)$ does not have closed range. This condition also appears naturally in relation to the existence of Lagrange multipliers, see e.g.\ \cite[§26.2]{Deimling1985}.
\end{remark}

\begin{remark}\label{remark:optimalrange}
Proposition \ref{prop:Jacobiannotsubmersion} does not hold when $p=2$, even when $T$ is linear. Indeed, consider the operator $T=(-\Delta)^{1/2}\circ \pi$, where $\pi\colon L^{2n}(\R^n,\R^n)\to L^2(\R^n)$ is the projection given by Proposition \ref{prop:subspacesofLp}\ref{it:p>2}. That $\pi$ can taken to be continuous follows from the fact that $L^2$ embeds complementably in $L^{2n}$; hence $T$ is continuous as well. The operator $T$ is clearly surjective and, being linear, is a submersion.

Curiously, a weaker version of Proposition \ref{prop:Jacobiannotsubmersion} holds in the case $p>2$: all operators $T\colon {\dot W}^{1,np}(\R^n,\R^n)\to L^p(\R^n)$ are \textit{regular} submersions nowhere. This follows from arguments similar to the ones above, using the last part of Proposition \ref{prop:subspacesofLp}.
\end{remark}

We note the following consequence of Proposition \ref{prop:Jacobiannotsubmersion} and the previous remark, which should be compared with Conjecture \ref{qu:Lp}:

\begin{cor}\label{cor:JacobianNondifferentiable}
Fix $p\in [1,+\infty)\exc\{2\}$. If $\J\colon {\dot W}^{1,np}(\R^n,\R^n)\to \mathscr H^p(\R^n)$ is onto then all of its right-inverses are Fréchet-differentiable nowhere.
\end{cor}

\begin{proof}
The statement follows from Proposition \ref{prop:Jacobiannotsubmersion} together with the chain rule. For the case $p>2$ we use the last part of Remark \ref{remark:optimalrange}, noting that if $\J$ has a right-inverse which is Fréchet-differentiable at $u_0$ then $\J'_{u_0}$ has a complemented kernel. 
\end{proof}

We conclude this subsection by discussing results related to Proposition \ref{prop:Jacobiannotsubmersion}. In the finite-dimensional case, the class of submersions is a good approximation for the class of open operators. This is made precise by the following proposition, which follows immediately from the Morse--Sard theorem:

\begin{prop}\label{prop:submersiondenseset}
Let $F\colon X\to Y$ be a smooth, surjective operator between two finite-dimensional vector spaces. The set of points where $F$ is a submersion is dense in $X$.
\end{prop}

Although there are versions of the Morse--Sard theorem in the infinite-dimensional setting, such a result fails completely in the absence of rather strong assumptions: for instance, a well-known version due to \textsc{Smale} \cite{Smale1965} requires the operator to have Fredholm derivatives. In our context, the failure of the infinite-dimensional Morse--Sard theorem, and consequently of Proposition \ref{prop:submersiondenseset}, is exemplified in a particularly striking way through the following result:

\begin{thm}\label{thm:smoothsurjection}
Take $p\in [1,+\infty)$ and let $Y$ be a separable Banach space. There is a smooth, locally Lipschitz, surjective operator $T\colon {\dot W}^{1,np}(\R^n,\R^n)\to Y$.
If $p>1$ then $T$ can be taken to additionally satisfy $\tp{rank}( T'_{u})\leq 1$ for all $u\in {\dot W}^{1,np}(\R^n,\R^n)$.
\end{thm}

For the proof of Theorem \ref{thm:smoothsurjection} we refer the reader to the work of \textsc{Bates} \cite{Bates1997}, for the $p>1$ case, as well as to \cite[Proposition 11.25]{Benyamini1999}.
Our interpretation of Theorem \ref{thm:smoothsurjection} is that the possible lack of surjectivity of the operator $\J\colon {\dot W}^{1,np}(\R^n,\R^n)\to \mathscr H^p(\R^n)$ cannot be proved by very general Banach space geometrical considerations in the spirit of this subsection.

\subsection{Existence of well-behaved solutions}
\label{sec:nicesols}

In this subsection we focus on the case $n=2$  for simplicity and we assume throughout that $\J \colon {\dot W}^{1,2p}(\R^2,\R^2)\to \mathscr H^p(\R^2)$ is surjective. 
We are particularly interested in the case $p=1$. Our goal is to illustrate the way in which Theorem \ref{thm:abstractopenmappingprinciple} yields the following principle:
\begin{quote}
\textit{the existence of rough solutions implies the existence of well-behaved solutions}.
\end{quote}
The following is an example a rough solution, and something that we would like to avoid:

\begin{ex}[\cite{Maly1995}]
There is a map $u\in W^{1,2}(\R^2,\R^2)$ such that 
$$\J u=0 \tp{ a.e.\ in }\R^2\qquad \tp{ and } \qquad u([0,1]\times \{0\})=[0,1]^2.$$ In particular, $u$ does not have the Lusin (N) property.
\end{ex}

The main result of this subsection is the following theorem, which shows that in some sense it suffices to deal with non-pathological solutions.

\begin{thm}
Let $\Omega\subset \R^2$ be a bounded open set and take $f\in \mathscr H^1(\R^2)$ such that $f\geq 0$ in $\Omega$. Assume that $\J\colon {\dot W}^{1,2}(\R^2,\R^2)\to \mathscr H^1(\R^2)$ is onto. Then there is a solution $u\in {\dot W}^{1,2}(\R^2,\R^2)$ of \eqref{eq:jac} such that:
\begin{enumerate}
\item\label{it:continuity} $u$ is continuous in $\Omega$;
\item\label{it:Lusin} $u$ has the Lusin (N) property in $\Omega$.
\item\label{it:apriori}$\int_{\R^2} |\D u|^2 \d x \le C \Vert f \Vert_{\mathscr{H}^1}$ with $C > 0$ independent of $f$.
\end{enumerate} 
In particular, $u$ satisfies the change of variables formula \eqref{eq:areaformula}.
Moreover, let $\Omega'\subseteq \Omega$ be an open set such that $f=0$ a.e.\ in $\Omega'$. Then:
\begin{enumerate}[resume]
\item\label{it:degeneratemonotonicity} for any set $E\subset \Omega'$, we have $u(\p E)=u(\overline E)$;
\item\label{it:componenttouchesbdry} for $y\in u(\Omega')$, if $C$ denotes a connected component of $u^{-1}(y)\cap \Omega'$ then $C$ intersects $\p \Omega'$.
\end{enumerate}
\end{thm}

Before proceeding with the proof, we note that \ref{it:degeneratemonotonicity} is a type of degenerate monotonicity which had already appeared in the study of the hyperbolic Monge--Ampère equation \cite{Chlebik2002, Kirchheim2001}.

\begin{proof}
The point of the proof is to perturb $f$ appropriately; then the solution $u$ is obtained as a limit of mappings of integrable distortion. 

Let $B^+$ be a ball containing $\Omega$ and let $B^-$ be another ball, disjoint from $\Omega$, and with the same volume as $B^+$. Consider the perturbations
$$f_\e \equiv  f + \e a, \qquad a\equiv \chi_{B^+}- \chi_{B^-},$$
which satisfy $f_\e > 0 $ a.e.\ in $\Omega$. Clearly $a\in \mathscr H^1(\R^2)$, being bounded, compactly supported and with zero mean. Hence $f_\e \to f$ in $\mathscr H^1(\R^2)$ and, from Corollary \ref{cor:openmapJacobian}, we see that we can choose solutions $u_\e$ of $\J u_\e = f_\e$ such that $\int_{\R^2} |\D u_\e|^2 \le C \Vert f_\e \Vert_{\mathscr{H}^1}$ for all $\e > 0$. Since the maps $u_\e$ have finite distortion, we can apply Theorem \ref{thm:theoryofFDmaps}\ref{it:continuity} to conclude that the family $(u_\e)$ is equicontinuous. Hence, upon normalising the maps so that $u_\e(x_0)=0$ for some fixed $x_0\in \Omega'$, and up to taking subsequences, $(u_\e)$ converges both locally uniformly in $\Omega$ and weakly in ${\dot W}^{1,2}(\R^2,\R^2)$ to a limit $u$. This already proves \ref{it:continuity} and \ref{it:apriori}.

To prove \ref{it:Lusin}, we note that each $u_\e$ satisfies \eqref{eq:pseudomonotone}, c.f.\ Proposition \ref{prop:FDimpliesmonotone}. Since $u$ is the uniform limit of the sequence $(u_\e)$, $u$ also satisfies \eqref{eq:pseudomonotone} and \ref{it:Lusin} follows from Proposition \ref{prop:pseudomonotonemappingsLusin}.

For \ref{it:degeneratemonotonicity}, note that $\e \leq f_\e$ in $\Omega$ and so each map $u_\e$, having integrable distortion, is open; it follows that $\p u_\e(E)\subseteq u_\e(\p E)$. Suppose, for the sake of contradiction, that there is $y\in u(\overline E)\exc u(\p E)$. On the one hand, there is some $\delta>0$ such that, for all $\e$ small enough,
$$B_\delta(y)\cap \p u_\e(\tp{int}\,E)\subset B_\delta(y)\cap u_\e(\p E)=\emptyset;$$
on the other hand, since $y\in u(\tp{int}\, E)$,
for all $\e$ small enough,
$$B_\delta(y) \cap u_\e(\tp{int}\,E)\neq \emptyset.$$
It follows that $B_\delta(y)\subseteq u_\e(\tp{int}\, E)$. 
We also have that $|u_\e(\tp{int}\, E)|\to 0$ as $\e \to 0$: by the change of variables formula,
$$|u_{\e}(\tp{int}\, E)|\leq \int_{u_{\e}(\tp{int}\, E)} \mc N(y, u_\e, \tp{int}\, E)\d y=\int_{E} \J u_{\e}= \e |E|\to 0.$$
Thus, since
$|B_\delta(y)|\leq |u_\e(E)|,$
a contradiction is reached by sending $\e\to 0$.

Finally, \ref{it:componenttouchesbdry} follows from \ref{it:degeneratemonotonicity}, as shown for instance in \cite[Lemma 2.10]{Kirchheim2001}.
\end{proof}

In view of the change of variables formula, it is useful to  control the multiplicity function. For the following proposition we again assume that the Jacobian is surjective.

\begin{prop}\label{prop:multiplicity}
Let $\Omega\subset \R^2$ be an open set and let $Y\equiv \{f\in \mathscr H^p(\R^2): f\geq c \tp{ a.e.\ in } \Omega\},$ where $c>0$.
Suppose that $f_j\in Y$ is a sequence converging weakly to $f$ in $\mathscr H^p(\R^2)$. 
For any maps $u_j\in {\dot W}^{1,2p}(\R^2,\R^2)$ satisfying $\J u_j=f_j$ and the a priori estimate \eqref{eq:aprioriestimateintro}, we have that
$$\sup_j\sup_{y\in u_j(\Omega')}  \mc N(y,u_j,\Omega')<\infty, \qquad \tp{ whenever }\Omega'\Subset \Omega.$$
\end{prop}

\begin{proof}
We claim that the sequence $u_j$ is equicontinuous and converges to $u\in {\dot W}^{1,2p}(\R^2,\R^2)$, a solution of $\J u =f$, uniformly in $\Omega'$. Once the claim is proved, the conclusion follows: $u$ has integrable distortion in $\Omega$ and so by Theorem \ref{thm:theoryofFDmaps}\ref{it:multiplicity} it is at most $m$-to-one in $\Omega'$, for some $m\in \N$. Thus, for all $j$ sufficiently large, $u_j$ is also at most $m$-to-one in $\Omega'$: if not, there are arbitrarily large $j$ and points $x_1^{(j)}, \dots, x_{m+1}^{(j)}\in \Omega'$ such that
$u_j(x_{i}^{(j)}) = y$ for some $y\in \R^n$ and all $i\in\{1,\dots,m+1\}$. By compactness, we can further assume that $x_i^{(j)}\to x_i$ for $i=1, \dots, m+1$. However, there are at least two different points $y_1\neq y_2$ such that
$$\{y_1, y_2\}\subset u(\{x_1, \dots, x_{m+1}\});$$
for the sake of definiteness, say $u(x_1)=y_1, u(x_2)=y_2$.
Let $\e < |y_1-y_2|$ and take $j$ sufficiently large so that, for $i=1,2,$
$$
|u_j(x_i^{(j)})-u_j(x_i)|<\frac \e 4, \qquad |u_j(x_i)-u(x_i)|<\frac \e 4;
$$
this is possible from equicontinuity of the sequence $u_j$ and the fact that it converges to $u$ uniformly. The triangle inequality gives $|y_1-y_2|=|u(x_1)-u(x_2)|<\e$, a contradiction.

To prove the claim, we assume that the Jacobian is surjective and we use Corollary \ref{cor:openmapJacobian}.
If $p>1$ we appeal to Morrey's inequality, $$[u_j ]_{C^{0,1-2/p}(\R^2)} \lesssim_p \Vert \D u_j \Vert_{L^{2p}(\R^2)}\leq C,$$
while for $p=1$ we use Theorem \ref{thm:theoryofFDmaps}\ref{it:continuity} instead. Either way, after normalizing the maps so that $u_j(x_0)=0$, where $x_0\in \Omega$, we see that the sequence $(u_j)$ is precompact in the local uniform topology over $\Omega'$. Hence we may assume that $u_j$ converges to some map $u\in {\dot W}^{1,2p}(\R^2,\R^2)$  uniformly in $\Omega'$ and also weakly in ${\dot W}^{1,2p}(\R^2,\R^2)$. 
\end{proof}

\section{A general nonlinear open mapping principle for scale-invariant problems} \label{sec:A general nonlinear open mapping principle for scale-invariant problems}
The  main result of this section is Theorem \ref{thm:abstractopenmappingprinciple2}, which is a generalisation of Theorem \ref{thm:abstractopenmappingprinciple} to a wider class of translation-invariant, scaling-invariant PDEs. Section \ref{sec:Examples} illustrates the way in which  Theorem \ref{thm:abstractopenmappingprinciple2} can be applied to some physical nonlinear equations: as particular examples, we consider the Navier--Stokes equations and the cubic wave equation.

\subsection{A more general nonlinear open mapping principle}
We begin by formulating a model problem abstractly as follows:
\begin{equation} \label{eq:abstractDirichlet}
\text{if } g-1 \in Y^* \text{, does there exist } v \text{ with } \tp{J} v = f \text{ and } v - \tp{id} \in X^*?
\end{equation}
Here $X^*$ and $Y^*$ are suitably chosen function spaces.
In smooth domains $\Omega \subsetneq \R^n$, examples of \eqref{eq:abstractDirichlet} include the Dirichlet problem for the Jacobian equation, that is, 
\begin{equation}
\begin{cases}\tp{J} v = g& \tp{in }\Omega,\\v = \tp{id}& \tp{on }\partial \Omega;\end{cases}
\label{eq:DirichletJacobian}
\end{equation} 
the condition $g-1 \in Y^*$ is codified in the compatibility condition
\[\int_\Omega (g-1) \d x = 0.\]

We return to the abstract formulation \eqref{eq:abstractDirichlet}. When $n = 2$, denoting $f \equiv g-1$ and $u \equiv v - \tp{id}$ we get the following question, equivalent to \eqref{eq:abstractDirichlet}:
\begin{equation} \label{eq:abstractDirichlet2}
\text{if } f \in Y^* \text{, does there exist } u \in X^* \text{ with } T u \equiv \tp{J} u + \tp{div} \, u = f?
\end{equation}
The latter formulation has the advantage that $X^*$ and $Y^*$ are vector spaces, which makes the problem more amenable to scaling arguments.

In Example \ref{ex:J+div} we discuss a representative special case of \eqref{eq:abstractDirichlet2}. Here $T$ does not map $X^*$ into $Y^*$ and we therefore need to choose a set $D \subsetneq X^*$ as the domain of definition of $T$.

\begin{ex}\label{ex:J+div}
Let $X^* = \dot{W}^{1,q}(\R^2,\R^2)$ and $Y^* = L^p(\R^2)$ with $p \in [2,\infty)$ and $q \in [p,2p]$. Since $T = \tp{J} + \tp{div}$ does not map $\dot{W}^{1,q}(\R^2,\R^2)$ into $L^p(\R^2)$, it is natural to set
\[D = \{u \in \dot{W}^{1,q}(\R^2,\R^2) \colon T u \in L^p(\R^2)\}\]
and study the range of
\begin{equation} \label{eq:Tindomains}
T = \tp{J} + \tp{div} \colon D \to Y^*.
\end{equation}
Note that we may write $T \circ \tau_\lambda^D = \tau_\lambda^{Y^*} \circ T$ for all $\lambda > 0$, where
\[\tau_\lambda^D u(x) = u_\lambda(x) \equiv \lambda u \left( \frac{x}{\lambda} \right), \qquad \tau_\lambda^{Y^*} f(x) = f_\lambda(x) \equiv f \left( \frac{x}{\lambda} \right)\]
give multiples of isometries:
\[\Vert \tau_\lambda^D u \Vert_{\dot{W}^{1,q}} = \lambda^{2/q} \Vert u \Vert_{\dot{W}^{1,q}}, \qquad \Vert \tau_\lambda^{Y^*} f \Vert_{L^p} = \lambda^{2/p} \Vert f \Vert_{L^p}\]
for all $u \in D$, $f \in Y^*$ and $\lambda > 0$.
\end{ex}

Since the set $D$ contains the proper dense subspace $C_c^\infty(\R^2,\R^2)$, it is neither weakly nor strongly closed in $\dot{W}^{1,q}(\R^2,\R^2)$. 
This difficulty is reflected in the somewhat awkward assumption \ref{it:weakseqclosed} of Theorem \ref{thm:abstractopenmappingprinciple2} below.

Before formulating the result recall that when a direct sum of Banach spaces $X = \oplus_{i=1}^I X_i$ is endowed with the norm $\Vert w \Vert_X \equiv \sum_{i=1}^I \Vert w_i \Vert_{X_i}$, the dual norm of $X^* = \oplus_{i=1}^I X_i^*$ is of the form $\Vert u \Vert_{X^*} = \max_{1 \le i \le I} \Vert u_i \Vert_{X_i^*}$.

\begin{thm} \label{thm:abstractopenmappingprinciple2}
Let $X_1,\ldots,X_I$ and $Y_1,\ldots,Y_J$ be Banach spaces and denote $X = \oplus_{i=1}^I X_i$ and $Y = \oplus_{j=1}^J Y_j$. Suppose $\mb B_{X^*}$ is sequentially weak$^*$ compact and $0 \in D \subset X^*$.

We make the following assumptions:
\begin{enumerate}
[label={\normalfont ($\widehat{\text{A}\arabic*}$)}]

\item \label{it:weakcontinuity2} 
$T\colon D\to Y^*$ has a weak$^*$-to-weak$^*$ closed graph: if $u_j \to u$ weak$^*$ and $T u_j\to f$ weak$^*$, then $Tu = f$.

\item \label{it:taus} For $\lambda>0$, there exist bijections $\tau_\lambda^D \colon D \to D$ and $\tau_\lambda^{Y^*} \colon Y^* \to Y^*$ such that
$$\begin{array}{ll}
 T \circ \tau_\lambda^D = \tau_\lambda^{Y^*} \circ T &  \text{for all } \lambda > 0, \\
 \Vert (\tau_\lambda^D u)_i \Vert_{X_i^*}  = \lambda^{r_i} \Vert u_i \Vert_{X_i^*} & \text{for all } \lambda > 0, \, i =1,\ldots,I, \, u \in X^*, \\
 \Vert (\tau_\lambda^{Y^*} f)_j \Vert_{Y_j^*} = \lambda^{s_j} \Vert f_j \Vert_{Y_j^*}\qquad & \text{for all } \lambda > 0, \, j =1,\ldots,J, \, f \in X^*,
\end{array}
$$
where $0 < r_1 \le \cdots \le r_I$ and $0 < s_1 \le \cdots \le s_J$.

\item \label{it:isometries2} There exist sequences of isometric bijections $\sigma_k^D \colon D \to D$ with $\sigma_k^D(0) = 0$ and isometric isomorphisms $\sigma_k^{Y^*} \colon Y^* \to Y^*$ such that
\begin{align*}
& T \circ \sigma_k^D = \sigma_k^{Y^*} \circ T \qquad \text{for all } k \in \N, 
&\sigma_k^{Y^*} f \overset{*}{\rightharpoonup} 0 \qquad \text{for all } f \in Y^*.
\end{align*}

\item \label{it:weakseqclosed} For $\ell \in \N$, the sets $D_\ell \equiv \{u \in D \colon \Vert u \Vert_{X^*} \le \ell, \, \Vert T u \Vert_{Y^*} \le \ell\}$ are weakly$^*$ sequentially closed in $X^*$.
\end{enumerate}

\noindent The following conditions are then equivalent:
\begin{enumerate}
\item \label{it:nonmeagre2} $T(D)$ is non-meagre in $Y^*$.

\item \label{it:surjectivity2} $T(D) = Y^*$.

\item \label{it:openat02} $T$ is open at the origin.

\item\label{it:quantitativesolvability2} For every $f \in Y^*$ there exists $u \in D$ such that
\begin{equation} \label{eq:aprioriestimate}
T u = f, \qquad \begin{cases}
\sum_{i=1}^I \Vert u_i \Vert_{X_i^*}^{s_J/r_i} \le C \Vert f \Vert_{Y^*}, & \Vert f \Vert_{Y^*} \le 1, \\
\sum_{i=1}^I \Vert u_i \Vert_{X_i^*}^{s_1/r_i} \le C \Vert f \Vert_{Y^*}, & \Vert f \Vert_{Y^*} > 1.
  \end{cases}
\end{equation}
\end{enumerate}
\end{thm}

\begin{proof}
We first show \ref{it:nonmeagre2} $\Rightarrow$ \ref{it:openat02}, so assume \ref{it:nonmeagre2} holds. Write $D = \cup_{\ell=1}^\infty D_\ell$ and note that $T(D) = \cup_{\ell=1}^\infty T(D_\ell)$. Since $\mb B_{X^*}$ is weak$^*$ sequentially compact and $T \colon D\to Y^*$ has weak$^*$-to-weak$^*$ closed graph, the sets $T(D_\ell)$ are closed in $Y$ and, therefore, complete. By the Baire Category Theorem, one of the sets $T(D_\ell)$ contains a ball $\bar{B}_\eta(f_0)$. Clearly $\eta \le \ell$. We first show that
\begin{equation} \label{eq:opennesspreliminary}
T(D \cap \ell \mathbb{B}_{X^*}) \supset \eta \mathbb{B}_{Y^*}.
\end{equation}

Suppose $f \in Y^*$ with $\Vert f \Vert_{Y^*} \le \eta$. Since the maps $\sigma_k^{Y^*} \colon Y^* \to Y^*$ are isometries, we get $f_0 + (\sigma_k^{Y^*})^{-1} f \in \bar{B}_\eta(f_0) \subset T(D_\ell)$ for every $k \in \N$. For every $k \in \N$, choose $u_k \in D_\ell$ such that $T u_k = f_0 + (\sigma_k^{Y^*})^{-1} f$. By \ref{it:isometries2}, each $\sigma_k^D$ maps $D_\ell$ into $D_\ell$. Thus, passing to a subsequence as in the proof of Theorem \ref{thm:abstractopenmappingprinciple}, and using \ref{it:weakseqclosed}, $\sigma_k^D u_k \rightharpoonup u \in D_\ell$; by \ref{it:weakcontinuity2} we get $T u = f$. Thus \eqref{eq:opennesspreliminary} is proved.

We are ready to show openness of $T$ at zero. Let $\e > 0$; our aim is to find $\delta > 0$ such that $T(D \cap \e
\mathbb{B}_{X^*}) \supset  \delta \mathbb{B}_{X^*}$. We first note that for each $\lambda > 0$ we have
\[\tau_\lambda^D (D \cap \ell \mathbb{B}_{X^*}) = \{u \in D \colon \Vert u_i \Vert_{X_i^*} \le \lambda^{r_i} \ell \text{ for } i =1,\ldots,I\}.\]
By choosing $\lambda = \min_{1 \le i \le I} (\e/\ell)^{1/r_i}$ we get $\max_{1 \le i \le I} \lambda^{r_i} \ell \le \e$ so that
\[T(D \cap \e \mathbb{B}_{X^*})
\supset T(\tau_\lambda^D(D \cap \ell \mathbb{B}_{X^*})) = \tau_\lambda^{Y^*} T(D \cap \ell \mathbb{B}_{X^*}).\]
By using \eqref{eq:opennesspreliminary} and selecting $\delta = \min_{1 \le i \le I} \min_{1 \le j \le J} \eta (\e/\ell)^{s_j/r_i}$ we get
\[\tau_\lambda^{Y^*} T(D \cap \ell \mathbb{B}_{X^*})
\supset \tau_\lambda^{Y^*}(\eta \mathbb{B}_{Y^*}) = \lambda^{s_1} \eta \mathbb{B}_{Y_1^*} \times \cdots \times \lambda^{s_J} \eta \mathbb{B}_{Y_J^*} \supset \delta \mathbb{B}_{Y^*},\]
as wished.

We now prove \ref{it:openat02} $\Rightarrow$ \ref{it:quantitativesolvability2}, so as above take some $\e>0$ and get $\delta>0$ in such a way that $\delta \mb B_{Y^*}\subset T(D\cap \e \mb B_{X^*})$. Assume, without loss of generality, that $\delta \le 1$. Let $f \in Y^*$ and define $\lambda > 0$ via
\[\Vert f \Vert_{Y^*} \equiv \mu = \min_{1 \le j \le J} \lambda^{s_j} \delta = \begin{cases}
\lambda^{s_J} \delta, & \mu \le \delta, \\
\lambda^{s_1} \delta, & \mu > \delta.
\end{cases}\]
In either case, let $j_0$ be such that $\mu = \lambda^{s_{j_0}} \delta$. Then
\begin{align*}
f
\in \tau_\lambda^{Y^*} (\delta \mathbb{B}_{Y^*})
\subset \tau_\lambda^{Y^*} T (D \cap \e \mathbb{B}_{X^*})
&= T \tau_\lambda^D (D \cap \e \mathbb{B}_{X^*}) \\
&= T \{u \in D \colon \Vert u_i \Vert_{X_i^*} \le \lambda^{r_i} \e \text{ for } i = 1,\ldots,I\}.
\end{align*}
Suppose now $u \in D$ satisfies $\Vert u_i \Vert_{X_i^*} \le \lambda^{r_i} \e$ for $i = 1,\ldots,I$. Then, for all $i$,
\[\Vert u_i \Vert_{X_i^*}^{s_{j_0}/r_i} \le \lambda^{s_{j_0}} \e^{s_{j_0}/r_i} \le \frac{\e^{s_{j_0}/r_i}}{\delta} \mu.\]
We conclude that
\[f \in T \left\{u \in D \colon \Vert u_i \Vert_{X_i^*} \le \lambda^{r_i} \e \text{ for all } i\right\}
\subset T \left\{ u \in D \colon \sum_{i=1}^I \Vert u_i \Vert_{X_i^*}^{s_{j_0}/r_i} \le C \mu \right\},\]
where
\[C = \sum_{i=1}^I \frac{\e^{s_{j_0}/r_i}}{\delta},\]
which yields \eqref{eq:aprioriestimate} in the cases $\Vert f \Vert_{Y^*} \le \delta$ and $\Vert f \Vert_{Y^*} > 1$. If ${ \Vert f \Vert_{Y^*} \in(\delta,1]}$, one obviously has $\lambda^{s_1} \approx_\delta \lambda^{s_J}$ so that \eqref{eq:aprioriestimate} holds for all $f$.

We conclude the proof of the theorem by noting that \ref{it:quantitativesolvability2} $\Rightarrow$ \ref{it:surjectivity2} $\Rightarrow$ \ref{it:nonmeagre2}.
\end{proof}

\begin{remark}\label{remark:cones}
Inspection of the proof of Theorem \ref{thm:abstractopenmappingprinciple2} shows that, in the statement of the theorem, one may replace all occurrences of $Y^*$ with $K$, where $K\subset Y^*$ is a closed convex cone. Recall that $K$ is said to be a \textit{cone} if $a f\in K$ whenever $a>0$ and $f \in K$.

Such a generalisation is occasionally useful, since it may be interesting to consider smaller data sets. For instance, in the case Question \ref{qu:Lp}, natural examples include the set of radially symmetric data $K = \{f \in \mathscr{H}^p(\R^n) : f(x) \equiv f(|x|)\}$ and, when $p > 1$, the set of non-negative data $K = \{f \in L^p(\R^n) : f \ge 0\}$. 
\end{remark}

Returning to Example \ref{ex:J+div}, it is easy to check that the assumptions of the theorem are satisfied, and so we may apply it to get the following:

\begin{cor}
Let $p \in [2,\infty)$ and $q \in [p,2p]$. The following claims are equivalent:
\begin{enumerate}
\item For all $f \in L^p(\R^2)$ there exists $u \in \dot{W}^{1,q}(\R^2,\R^2)$ with $\tp{J} u + \tp{div} \, u = f$.

\item For all $f \in L^p(\R^2)$ there exists $u \in \dot{W}^{1,q}(\R^2,\R^2)$ with
\[\tp{J} u + \tp{div} \, u = f, \qquad \Vert \D u \Vert_{L^q}^q \le C \Vert f \Vert_{L^p}^p.\]
\end{enumerate}
\end{cor}

\begin{remark}
When $\Omega \subset \R^n$ is a smooth, bounded domain, $n \le p  < \infty$, $X = W_0^{1,p}(\Omega,\R^2)$ and $Y^* = L^p(\Omega)/\R$, the question about surjectivity of the operator $T = \tp{div} + \tp{J}$ is closely related to \cite[Question 1.2]{GuerraKochLindberg2020}. It would be interesting to find out whether Theorems \ref{thm:abstractopenmappingprinciple} and \ref{thm:abstractopenmappingprinciple2} can be adapted to bounded domains.
\end{remark}

\subsection{A version for inhomogeneous function spaces}
We are also interested in applying Theorem \ref{thm:abstractopenmappingprinciple2} to inhomogeneous function spaces. In order to achieve this we first recall some definitions from interpolation theory:
\begin{ndef}
Suppose that $X_1$ and $X_2$ are Banach spaces embed into a topological vector space $Z$. We set
\begin{align*}
& \Vert u \Vert_{X_1 \cap X_2} \equiv \max \{\Vert u \Vert_{X_1}, \Vert u \Vert_{X_2}\}, \\
& \Vert u \Vert_{X_1 + X_2} \equiv \inf\{\Vert u_1 \Vert_{X_1} + \Vert u_2 \Vert_{X_2} \colon u = u_1 + u_2, \, u_1 \in X_1, \, u_2 \in X_2\}.
\end{align*}
If $X_1 \cap X_2$ is dense in both $X_1$ and $X_2$, then $(X_1,X_2)$ is called a \emph{conjugate couple}.
\end{ndef}
The duals of spaces of the form $X_1\cap X_2$ are well-known, c.f.\ \cite[Theorem 2.7.1]{Bergh1976}:

\begin{thm}
Let $(X_1,X_2)$ be a conjugate couple. Then, up to isometric isomorphism, $(X_1 \cap X_2)^* = X_1^* + X_2^*$ and $(X_1+X_2)^* = X_1^* \cap X_2^*$.
\end{thm}

Following the proof of Theorem \ref{thm:abstractopenmappingprinciple2} almost verbatim we obtain the following:
\begin{cor}\label{cor:abstractopenmappingprincipleintersections}
For $i=1,\ldots I$, $j=1,\ldots J_i$ and $\mu=1,\ldots M$, $\nu=1,\ldots N_\mu$ let $X_{i,j}$ and $Y_{\mu,\nu}$ be Banach spaces.
Consider $X^*,Y^*$ of the form $$X^* = \bigoplus_{i=1}^I \bigg(\bigcap_{j=1}^{J_i} X_{i,j}^*\bigg), \qquad Y^* = \bigoplus_{\mu=1}^M \bigg(\bigcap_{\nu=1}^{N_\mu} Y_{\mu,\nu}^*\bigg)$$
for some $I,M,J_i,N_\mu\in \N$. 
 
Suppose assumptions \ref{it:weakcontinuity2}, \ref{it:isometries2} and \ref{it:weakseqclosed} of Theorem \ref{thm:abstractopenmappingprinciple2} hold. Suppose further that for $\lambda>0$, there exist bijections $\tau_\lambda^D \colon D \to D$ and $\tau_\lambda^{Y^*} \colon Y^* \to Y^*$ such that
$$\begin{array}{ll}
 T \circ \tau_\lambda^D = \tau_\lambda^{Y^*} \circ T &  \text{for all } \lambda > 0, \\
 \Vert (\tau_\lambda^D u)_i \Vert_{X_{i,j}^*}  = \lambda^{r_{i,j}} \Vert u_i \Vert_{X_{i,j}^*} & \text{for all } \lambda > 0, \, i =1,\ldots,I, \, j=1,\ldots,J_i, \, u \in X^*, \\
 \Vert (\tau_\lambda^{Y^*} f)_\mu \Vert_{Y_{\mu,\nu}^*} = \lambda^{s_{\mu,\nu}} \Vert f_\mu \Vert_{Y_{\mu,\nu}^*}\qquad & \text{for all } \lambda > 0, \, \mu =1,\ldots,M, \, \nu=1,\ldots,N_\mu,\, f \in X^*,
\end{array}
$$
where $0<r_{i,j}$ and $0 < s_1 \le s_{\mu,\nu} \le s_2$. 

\noindent Then the following conditions are equivalent:
\begin{enumerate}
\item $T(D)$ is non-meagre in $Y^*$.

\item $T(D) = Y^*$.

\item $T$ is open at the origin.

\item For every $f \in Y^*$ there exists $u \in D$ such that
\begin{equation*}
T u = f, \qquad \begin{cases}
\sum_{i=1}^I\sum_{j=1}^{J_i} \Vert u_i \Vert_{X_{i,j}^*}^{s_2/r_{i,j}} \le C \Vert f \Vert_{Y^*}, & \Vert f \Vert_{Y^*} \le 1, \\
\sum_{i=1}^I\sum_{j=1}^{J_i} \Vert u_i \Vert_{X_{i,j}^*}^{s_1/r_{i,j}} \le C \Vert f \Vert_{Y^*}, & \Vert f \Vert_{Y^*} > 1.
  \end{cases}
\end{equation*}
\end{enumerate}
\end{cor}

\subsection{Two model examples} \label{sec:Examples}

In this subsection we illustrate the use of Theorem \ref{thm:abstractopenmappingprinciple2} in the model cases of the 3D Navier-Stokes equations and the 3D cubic wave equation. 

\begin{ex}
We illustrate the use of Theorem \ref{thm:abstractopenmappingprinciple2} in the model case of the homogeneous, incompressible Navier-Stokes equations in $\R^3 \times [0,\infty)$:
\begin{align}
& \partial_t u + u \cdot \nabla u - \nu \Delta u - \nabla P = 0, \label{eq:NS1} \\
& \tp{div}\, u = 0, \label{eq:NS2} \\
& u(\cdot,0) = u^0, \label{eq:NS3}
\end{align}
where $u$ is the velocity field, $P$ is the pressure, $\nu>0$ is the viscosity and $u^0$ is the initial data. The equations are invariant under the scalings $u \to u_\lambda$, $P \to P_\lambda$ and $u^0 \to u^0_\lambda$,
\[u_\lambda(x,t) \equiv \frac{1}{\lambda} u \left( \frac{x}{\lambda},\frac{t}{\lambda^2} \right), \qquad P_\lambda(x,t) \equiv \frac{1}{\lambda^2} P \left( \frac{x}{\lambda}, \frac{t}{\lambda^2} \right), \qquad u^0_\lambda(x) = \frac{1}{\lambda} u^0\left( \frac{x}{\lambda} \right).\]
We divide the discussion into the following three steps: i) formally determining $T$; ii) choosing relevant ambient spaces $X^*$ and $Y^*$; iii) choosing the domain of definition $D$.

We begin by choosing the operator $T$ we wish to study. We incorporate \eqref{eq:NS1}--\eqref{eq:NS2} into the choice of the function spaces and choose, formally, $T(u) = u(\cdot,0)$. As the sought range we consider $Y^* = L^2_\sigma = \{v \in L^2(\R^3,\R^3) \colon \tp{div}\, v = 0\}$. We wish to choose the domain of definition $D$ to be a suitable set of functions which satisfy \eqref{eq:NS1}--\eqref{eq:NS3} for some $u^0 \in L^2_\sigma$. We also need to determine the ambient space $X^*$.

In order for Theorem \ref{thm:abstractopenmappingprinciple2} to be applicable, we wish to consider regularity regimes where $T$ has a weak$^*$-to-weak$^*$ closed graph and the sets $D_\ell = \{u \in D \colon \Vert u \Vert_X \le \ell, \, \Vert T u \Vert_{Y^*} \le \ell\}$ are weakly$^*$ compact. It is natural to set $X^* = L^p_t (L^q_\sigma)_x \cap L^r_t \dot{W}^{1,s}_x$ for suitable $p,q,r,s \in [1,\infty]$. For condition \ref{it:taus} of Theorem \ref{thm:abstractopenmappingprinciple2} we compute, for all $p,q,r,s\in [1,\infty]$,
\begin{align*}
& \Vert u_\lambda^0 \Vert_{L^q} = \lambda^{3/q-1} \Vert u^0 \Vert_{L^q}, \\
& \Vert u_\lambda \Vert_{L^p_t L^q_x} = \lambda^{2/p+3/q-1} \Vert u\Vert_{L^p_t L^q_x}, \\
& \Vert u_\lambda \Vert_{L^r_t \dot{W}^{1,s}_x} = \lambda^{2/r+3/s-2} \Vert u \Vert_{L^r_t \dot{W}^{1,s}_x}.
\end{align*}
Thus \ref{it:taus} requires the compatibility condition $2/p+3/q-1=2/r+3/s-2$ to hold.

For simplicity, we consider the most familiar choice of exponents, that is we consider {$X^* = L^\infty_t L^2_{\sigma,x} \cap L^2_t \dot{W}^{1,2}_x$}. Recall that $u \in X$ is called a \emph{weak solution} of \eqref{eq:NS1}--\eqref{eq:NS3} if $u$ satisfies
\begin{equation} \label{eq:weakNS}
\int_0^\tau \langle u, \partial_t \varphi \rangle \d t + \int_0^\tau \langle u \otimes u, \D \varphi \rangle \d t - \nu \int_0^\tau \langle \D u, \D \varphi \rangle \d t + \langle u^0, \varphi(0) \rangle - \langle u(\tau), \varphi(\tau) \rangle = 0
\end{equation}
for all $\varphi \in C_c^\infty(\R^3 \times [0,\infty),\R^3)$ with $\tp{div}\, \varphi = 0$ and almost every $\tau > 0$. In \eqref{eq:weakNS}, $\langle\cdot,\cdot\rangle$ denotes the inner product in $L^2_x$. This prompts us to set
\begin{gather*}
 D \equiv \{u \in  L^\infty_t L^2_{\sigma,x} \cap L^2_t \dot{W}^{1,2}_x : u \text{ is a weak solution of \eqref{eq:NS1}--\eqref{eq:NS3} for some } u^0 \in L^2_\sigma\}, \\
 T \colon D \to Y^*, \quad T(u) \equiv u^0 \text{ if \eqref{eq:weakNS} holds.}
\end{gather*}
We briefly indicate why $T \colon (D,\tp{wk}^*) \to (Y^*,\tp{wk}^*)$ has weak$^*$-to-weak$^*$ closed graph and why the sets $D_\ell$ are weakly$^*$ closed for all $\ell \in \N$.
When $u \in D$, we have $\partial_t u \in L^{4/3}(0,\tau,(W^{1,2}_\sigma)^*)$ for all $\tau > 0$ (see \cite[Lemma 3.7]{Robinson2016}). Thus, by using the Aubin--Lions lemma and a diagonal argument, if $u_j \overset{*}{\rightharpoonup} u$ in $D$, then every subsequence has a subsequence converging strongly in $L^2(0,\tau,L^2(B_R,\R^3))$ for all $\tau,R > 0$. The strong convergence and \eqref{eq:weakNS} imply that every subsequence of $(u^0_j)_{j \in \N}$ has a subsequence converging weakly$^*$ to $u^0$. This implies the two claims made above.

Theorem \ref{thm:abstractopenmappingprinciple2} now says that solvability of \eqref{eq:NS1}--\eqref{eq:NS3} for all $u^0 \in L^2_\sigma$ is equivalent to solvability with the \emph{a priori} estimate $$\Vert u \Vert_{L^\infty_t L^2_x} + \Vert u \Vert_{L^2_t \dot{W}^{1,2}_x} \le C \Vert u(\cdot,0) \Vert_{L^2}.$$ Such an estimate is satisfied by Leray--Hopf solutions \cite{Robinson2016}.
\end{ex}

\begin{ex}
Consider the cubic wave equation in $(1+3)$-dimensions
\begin{align}\label{eq:3dCubicWave}
\p_{tt} u - \Delta u + u^3 = 0 \text{ in } [0,+\infty)\times\R^3\\
(u(\cdot,0),\p_t u(\cdot,0)) = (u^0,u^1).
\label{eq:3dCubicWaveinitialdata}
\end{align}
We are interested in initial data in the energy space $Y^*= [\dot H^1(\R^3)\cap L^{4}(\R^3)]\times L^2(\R^3)$ and
we look for solutions in the space 
$$X^* = L^\infty_t \dot H^1_x \cap L^\infty_t L^{4}_x\cap L^\rho_t L^\sigma_x([0,\infty)\times \R^3), \qquad \tp{where } \frac 1 \rho+\frac 3 \sigma=\frac 1 2.$$ In the notation of Corollary \ref{cor:abstractopenmappingprincipleintersections}, we have $X_{1,1}^*=L^\infty_t \dot H^1_x$, $X_{1,2}^*=L^\infty_t L^4_x$, $X_{1,3}^*=L^\rho_t L^\sigma_x$ and $Y_{1,1}^*=\dot H^1$, $Y^*_{1,2}=L^4$ and $Y_{2,1}^*=L^2$.
Recall that $u\in X^*$ is a weak solution of \eqref{eq:3dCubicWave}--\eqref{eq:3dCubicWaveinitialdata} if, for every test function ${\varphi\in C_c^\infty(\R\times \R^3,\R^3)}$, \begin{equation}
\begin{split}\label{eq:3dCubicWaveWeak}
\int_0^\tau u\, \p_{tt} \varphi +\langle \D u, \D\varphi\rangle\d t+ \langle u^3,\varphi\rangle \d t & =\langle u^0,\varphi(0,\cdot)\rangle-\langle u^1,\varphi_t(0,\cdot)\rangle\\ & \quad -\langle u(\tau,\cdot),\varphi(\tau,\cdot)\rangle+\langle u(\tau,\cdot),\varphi_t(\tau,\cdot)\rangle
\end{split}
\end{equation}
for almost every $\tau >0$.
The equation is invariant under translations as well as the scalings $u\to u_\lambda, u^0\to u^0_\lambda$ and $u^1 \to u^1_\lambda$ where
\begin{align*}
u_\lambda(t,x) \equiv \lambda u(\lambda t, \lambda x) \qquad u^0_\lambda(x) \equiv \lambda u^0_\lambda(\lambda x)\qquad u^1_\lambda(x)\equiv\lambda^2 u^1_\lambda(\lambda^2 x).
\end{align*}
We formally define
\begin{gather*}
D = \{u\in X^*\colon u \text{ is a weak solution of  \eqref{eq:3dCubicWave}--\eqref{eq:3dCubicWaveinitialdata} for some } (u^0,u^1)\in Y^*\},\\
T\colon C\to Y^*, Tu = (u^0,u^1) \text{ if } \eqref{eq:3dCubicWaveWeak} \text{ holds}.
\end{gather*}
It is easy to compute
\begin{align*}
\|u^0_\lambda\|_{L^4} =& \lambda^{\frac 1 4}\|u^0\|_{L^4},\qquad& \|(u^0_\lambda,u^1_\lambda)\|_{\dot H^1\times L^2}=& \lambda^{\frac 1 2}\|(u^0,u^1)\|_{\dot H^1\times L^2}, \\
\|u_\lambda\|_{L^\infty_t L^4_x} =& \lambda^{\frac 1 4}\|u\|_{L^\infty_t L^4_x},\qquad&\|u_\lambda\|_{L^\infty_t \dot H^1_x\cap L^\rho_t L^\sigma_x} =& \lambda^{\frac 1 2}\|u\|_{L^\infty_t \dot H^1_x\cap L^\rho_t L^\sigma_x}.
\end{align*}
Thus $r_{11}=r_{13}=\frac 1 2$, $r_{12}=\frac 1 4$ and $s_{11}=s_{21}=\frac 1 2$, $s_{12}=\frac 1 4$ in the notation of Corollary \ref{cor:abstractopenmappingprincipleintersections}.

The subcritical nature of \eqref{eq:3dCubicWave} implies that $T \colon (D,\tp{wk}^*) \to (Y^*,\tp{wk}^*)$ has a closed graph and that the sets $D_\ell$ are weakly$^*$ closed for all $\ell \in \N$.
Indeed, when $u \in D$, it is not difficult to see that $\p_t u\in L^\infty(0,\tau;(\dot H^1\cap L^4)^*)$ for all $\tau > 0$: it suffices to test the weak formulation \eqref{eq:3dCubicWaveWeak} against functions of the type $\varphi(t,x)=\a(t)\phi(x)$ where both $\alpha$ and $\phi$ are test functions and $\a\equiv 1$ on a given time interval $[t_1,t_2]$. Thus, by using the Aubin--Lions lemma and a diagonal argument, if $u_j \overset{*}{\rightharpoonup} u$ in $D$, then every subsequence has a subsequence converging strongly in $L^2(0,\tau,L^{4}(B_R,\R^3))$ for all $\tau,R > 0$. The strong convergence and \eqref{eq:3dCubicWaveWeak} imply that every subsequence of $(u^0_j,u^1_j)_{j \in \N}$ has a subsequence converging weakly$^*$ to $(u^0,u^1)$. This in turn implies the two claims made above.

Corollary \ref{cor:abstractopenmappingprincipleintersections} now says that solvability of \eqref{eq:3dCubicWave}--\eqref{eq:3dCubicWaveinitialdata} for all $(u^0,u^1) \in [\dot{H}^1 \cap L^4] \times L^2$ is equivalent to solvability with the \emph{a priori} estimate 
$$
\begin{cases}
\Vert u \Vert_{L^\infty_t \dot H^1_x\cap L^\rho_t L^\sigma_x}^{\frac 1 2}+\|u\|_{L^\infty_t L^{4}_x} \leq C \|(u^0,u^1)\|_{\dot H^1 \cap L^4\times L^2}, & \tp{if }\|(u^0,u^1)\|_{\dot H^1 \cap L^{4}\times L^2}\le 1, \\
 \Vert u \Vert_{L^\infty_t \dot H^1\cap L^\rho_t L^\sigma_x}+\|u\|_{L^\infty_t L^4_x}^{2} \le C \|(u^0,u^1)\|_{\dot H^1 \cap L^4\times L^2}, & \tp{if } \|(u^0,u^1)\|_{\dot H^1 \cap L^{4}\times L^2}> 1.
  \end{cases}
$$
Taking powers and estimating the right-hand sides, we conclude in particular that solvability of the equation implies the more familiar-looking estimate
\begin{align*}
\|u\|_{L^\infty_t \dot H^1_x\cap L^\rho_t L^\sigma_x}^2 + \|u\|_{L^\infty_t L^4_x}^{4}\leq C \left(\|u^0\|_{\dot H^1}^2+\|u^0\|_{L^{4}}^{4}+\|u^1\|_{L^2}^2\right).
\end{align*}
The estimate in the Strichartz space $L^\rho_t L^\sigma_x$ is known from \cite{Ginibre1985}, and the reader may also find the stronger estimate
\[\frac 1 2 \|u\|_{L^\infty_t \dot H^1_x}^2 + \frac 1 4\|u\|_{L^\infty_t L^4_x}^{4}\leq \frac 1 2 \|u^0\|_{\dot H^1}^2+\frac 1 4\|u^0\|_{L^{4}}^{4}\]
in \cite[Theorem 8.41]{Bahouri2011}.
\end{ex}

\section{Non-surjectivity under incompatible scalings} \label{sec:Non-surjectivity under incompatible scalings}

Theorem \ref{thm:abstractopenmappingprinciple2} is useful in proving non-solvability of various problems which admit multiple scaling symmetries. The main  result of this section, Theorem \ref{thm:linearproblems}, encapsulates this idea. We then apply Theorem \ref{thm:linearproblems} to several examples, such as the Jacobian and the incompressible Euler and Navier--Stokes equations, proving in particular Corollary \ref{thm:Eulertheorem}. We note that, for these applications, the full \textit{nonlinear} strength of Theorem \ref{thm:linearproblems} is not always needed, as often we can relax nonlinear PDEs into linear ones. Besides being useful to prove non-solvability, this strategy also gives an elementary way of proving upper bounds on the energy dissipation rates for Baire-generic initial data in evolutionary PDEs.

\subsection{A general non-solvability result}

In Theorem \ref{thm:abstractopenmappingprinciple2}, surjectivity can only hold if all the scaling symmetries are compatible. The following result makes this precise.

\begin{thm} \label{thm:linearproblems}
Consider the setup and assumptions of Theorem \ref{thm:abstractopenmappingprinciple2}, with $J = 1$. Suppose, additionally, that there exists other bijections $\tilde{\tau}_\lambda^D$  and $\tilde{\tau}_\lambda^{Y^*}$ satisfying \ref{it:taus}:
$$\begin{array}{ll}
 T \circ \tilde{\tau}_\lambda^D = \tilde{\tau}_\lambda^{Y^*} \circ T &  \text{for all } \lambda > 0, \\
 \Vert (\tilde{\tau}_\lambda^D u)_i \Vert_{X_i^*}  = \lambda^{\tilde{r}_i} \Vert u_i \Vert_{X_i^*} & \text{for all } \lambda > 0, \, i =1,\ldots,I, \, u \in X^*, \\
 \Vert \tilde{\tau}_\lambda^{Y^*} f \Vert_{Y^*} = \lambda^{\tilde{s}} \Vert f \Vert_{Y^*}\qquad & \text{for all } \lambda > 0, \, f \in Y^*,
\end{array}
$$
where $0 < \tilde{r}_1 \le \cdots \le \tilde{r}_I$ and $\tilde{s} > 0$. If
\[\frac{\tilde{s}}{\tilde{r_i}} > \frac{s}{r_i} \text{ for all } i,\]
then $T(M \mathbb{B}_{X^*})$ is nowhere dense in $Y^*$ for every $M > 0$.
\end{thm}

\begin{proof}
Seeking a contradiction, assume $T(M \mathbb{B}_{X^*})$ is not nowhere dense. By weak$^*$ sequential compactness of $\mathbb{B}_{X^*}$ and since $T$ has a weak$^*$-to-weak$^*$ closed graph, the set $T(M \mathbb{B}_{X^*})$ is closed, and thus $T(M \mathbb{B}_{X^*})$ contains a ball $\bar{B}_r(f_0)$. By Theorem \ref{thm:abstractopenmappingprinciple2}, for every $f \in Y^*$ there exists $u \in X^*$ such that $T u = f$ and $\sum_{i=1}^I \Vert u_i \Vert_{X_i^*}^{s/r_i} \le C \Vert f \Vert_{Y^*}$ and $\tilde{u} \in X^*$ such that $T \tilde{u} = f$ and $\sum_{i=1}^I \Vert \tilde{u}_i \Vert_{X_i^*}^{\tilde{s}/\tilde{r}_i} \le C \Vert f \Vert_{Y^*}$.

Fix $f \in Y^*$ with $\Vert f \Vert_{Y^*} = 1$; our aim is to show that $T 0 = f$ and derive a contradiction. First note that $\Vert \tau_\lambda^{Y^*} f \Vert_{Y^*} = \lambda^s$. Choose $\tilde{u} \in X^*$ with $T \tilde{u} = \tau_\lambda^{Y^*} f$ and $\sum_{i=1}^I \Vert \tilde{u}_i \Vert_{X_i^*}^{\tilde{s}/\tilde{r}_i} \le C \lambda^s$. Write $\tilde{u} = \tau_\lambda^D u$, so that
\begin{align*}
& T u = T (\tau_\lambda^D)^{-1} \tilde{u} = (\tau_\lambda^{Y^*})^{-1} T \tilde{u} = f, \\
& \sum_{i=1}^I \lambda^{r_i \tilde{s}/\tilde{r}_i} \Vert u_i \Vert_{X_i^*}^{\tilde{s}/\tilde{r}_i} = \sum_{i=1}^I \Vert \tilde{u}_i \Vert_{X_i^*}^{\tilde{s}/\tilde{r}_i} \le C \lambda^s.
\end{align*}
We conclude that $\sum_{i=1}^I \lambda^{r_i\tilde{s}/\tilde{r}_i - s} \Vert u_i \Vert_{X_i^*}^{\tilde{s}/\tilde{r}_i} \le \tilde{C}$. Thus, by letting $\lambda \to 0$ or $\lambda \to \infty$ we find a sequence of solutions $u^\ell$ of $T u^\ell = f$ with $\sum_{i=1}^I \Vert u_i^\ell \Vert_{X_i^*}^{\tilde{s}/\tilde{r}_i} \to 0$. Now $\Vert u^\ell \Vert_{X^*} \to 0$ so that $u^\ell \overset{*}{\rightharpoonup} 0$, which yields $T 0 = f$. Thus $1 = \Vert T 0 \Vert_{Y^*} = \Vert T \tau_\lambda^D 0 \Vert_{Y^*} = \Vert \tau_\lambda^{Y^*} T 0 \Vert_{Y^*} = \lambda^s$ for all $\lambda > 0$. We have reached a contradiction.
\end{proof}

\begin{remark} \label{rem:incompatible inhomogeneous scalings}
The conclusion of Theorem \ref{thm:linearproblems} also follows if $\tilde{s}/\tilde{r}_i \ge s/r_i$ for all $i$ and $\tilde{s}/\tilde{r}_{i_0} > s/r_{i_0}$ for some $i_0 \in \{1,\ldots,I\}$ such that $T(u_1,\ldots,u_{i_0-1},0,u_{i_0+1},\ldots,u_I) \equiv 0$. The conclusion also follows if $X^* = \cap_{i=1}^I X_i^*$ instead of $X^* = \oplus_{i=1}^I X_i^*$ and $\Vert u_{i_0} \Vert_{X_{i_0}^*} = 0$ implies $T u = 0$. We illustrate the latter point below by recovering the main result of~\cite{Lindberg2017} (albeit in a weaker form).
\end{remark}

\begin{cor}\label{cor:jacobianinhomogeneous}
If $n \ge 2$ and $p \in [1,\infty)$, the set $\J(W^{1,np}(\R^n,\R^n))$ is meagre in $\mathscr{H}^p(\R^n)$.
\end{cor}

\begin{proof}
Write $W^{1,np}(\R^n,\R^n) = L^{np}(\R^n,\R^n) \cap \dot{W}^{1,np}(\R^n,\R^n)$. The scaling $u_\lambda = \lambda u(\cdot/\lambda)$, $f_\lambda = f(\cdot/\lambda)$, under which $\J$ is invariant, gives
\[\Vert u_\lambda \Vert_{L^{np}} = \lambda^{1+1/p} \Vert u \Vert_{L^{np}}, \quad \Vert u_\lambda \Vert_{\dot{W}^{1,np}} = \lambda^{1/p} \Vert u \Vert_{\dot{W}^{1,np}}, \quad \Vert f \Vert_{\mathscr{H}^p} = \lambda^{n/p} \Vert f \Vert_{\mathscr{H}^p},\]
so that $s/r_1 = n/(p+1)$ and $s/r_2 = n$, whereas the scalings $\tilde{\tau}_\lambda^D \equiv \lambda \, \tp{id}$ and $\tilde{\tau}_\lambda^{Y^*} = \lambda^n \, \tp{id}$ give $\tilde{s}/\tilde{r}_1 = \tilde{s}/\tilde{r}_2 = n$. The claim follows from Remark \ref{rem:incompatible inhomogeneous scalings} since $\Vert u \Vert_{L^{np}} = 0$ and $\Vert u \Vert_{\dot{W}^{1,np}} < \infty$ yield $\tp{J} u = 0$.
\end{proof}

As another example we consider the linear, homogeneous heat equation with $L^2$ data:
\[\partial_t u - \nu \Delta u = 0 \quad \tp{ in } \R^3\times (0,+\infty), \qquad u(\cdot,0) = u^0.\]
Theorem \ref{thm:linearproblems} effortlessly yields the following essentially classical result:

\begin{cor}
Let $1 < p < \infty$ and $M > 0$. The set of data $u^0 \in L^2(\R^3)$ with a solution satisfying $\Vert u \Vert_{L^p_t L^2_x} \le M$ is nowhere dense in $L^2(\R^3)$. In particular, a Baire-generic datum $u^0 \in L^2(\R^3)$ does not have a solution $u \in L^p_t L^2_x$ for any $p \in (1,\infty)$.
\end{cor}

\begin{proof}
Let $1 < p < \infty$. We set $D = \{u \in L^p_t L^2_x \colon u(\cdot,0) = u^0 \text{ for some } u^0 \in L^2_\sigma\}$ and $T(u) \equiv u^0$. The claim follows from Theorem \ref{thm:linearproblems}. As one scaling we use the parabolic one:
\[u_\lambda(x,t) = \frac{1}{\lambda} u \left( \frac{x}{\lambda}, \frac{t}{\lambda^2} \right),\]
so that that $\Vert u_\lambda \Vert_{L^p_t L^2_x} = \lambda^{1/2+2/p} \Vert u \Vert_{L^p_t L^2_x}$ and $\Vert u_\lambda^0 \Vert_{L^2} = \lambda^{1/2}$. As the other scaling we use $\tilde{\tau}_\lambda^D \equiv \lambda \, \tp{id}$ and $\tilde{\tau}_\lambda^{Y^*} \equiv \lambda\, \tp{id}$; now $s/r = 1/(1+4/p) < 1/1 = \tilde{s}/\tilde{r}$.
\end{proof}

\subsection{The incompressible Euler equations and the proof of Corollary \ref{thm:Eulertheorem}} \label{sec:The incompressible Euler equations}
Our next aim is to prove Corollary \ref{thm:Eulertheorem} on the incompressible Euler equations in $\R^n \times [0,\infty)$, $n \ge 2$. Recall that given $u^0 \in L^2_\sigma$, a mapping $u \in L^p_t L^2_{\sigma,x}$, $2 \le p \le \infty$, is a weak solution of the Cauchy problem \eqref{eq:Euler1}--\eqref{eq:Euler3} if
\begin{equation} \label{eq:Eulerweakformulation}
\int_0^\infty \int_{\R^n} (u \cdot \partial_t \varphi + u \otimes u : \D \varphi) \d x \d t + \int_{\R^n} u^0 \cdot \varphi(\cdot,0) \d x = 0 \quad \forall \varphi \in C_{c,\sigma}^\infty(\R^n \times [0,\infty), \R^n).
\end{equation}

We cannot deduce Corollary \ref{thm:Eulertheorem} directly via Theorem \ref{thm:abstractopenmappingprinciple2}. Indeed, the integral condition \eqref{eq:Eulerweakformulation} leads to a well defined mapping $T$ from a weak solution $u \in L^p_t L^2_{\sigma,x}$ of \eqref{eq:Euler1}--\eqref{eq:Euler3} to the initial data $u^0 \in L^2_\sigma$ but does not easily lend itself to a domain of definition $D \subset L^p_t L^2_{\sigma,x}$ satisfying condition \ref{it:weakseqclosed} of Theorem \ref{thm:abstractopenmappingprinciple2}. We therefore consider a \textit{relaxed problem} where $u \otimes u \in L^{p/2}_t L^1_x$ is replaced by a general matrix-valued mapping $S$.

In order to apply Theorem \ref{thm:linearproblems} we embed $L^1(\R^n,\R^{n \times n})$ into the space of signed Radon measures $\mathbf{M}(\R^n,\R^{n \times n})$ which is the dual of the separable Banach space $C_0(\R^n,\R^{n \times n})$. We endow $\mathbf{M}(\R^n,\R^{n \times n})$ with the dual norm. In the relaxed problem we require $u \in L^p_t L^2_{\sigma,x}$ and $S \in L^{p/2}_t \mathbf{M}_x$ to satisfy
\begin{equation} \label{eq:relaxedEuler}
\int_0^\infty \int_{\R^n} (u \cdot \partial_t \varphi + S : \D \varphi) \d x \d t + \int_{\R^n} u^0 \cdot \varphi(\cdot,0) \d x = 0 \quad \forall \varphi \in C_{c,\sigma}^\infty(\R^n \times [0,\infty), \R^n).
\end{equation}
Unlike \eqref{eq:Eulerweakformulation}, \emph{due to linearity, condition \eqref{eq:relaxedEuler} is stable under weak$^*$ convergence}.

Relaxations such as \eqref{eq:relaxedEuler} are studied in \emph{Tartar's framework}, where a system of nonlinear PDEs is decoupled into a set of linear PDEs (conservation laws) and pointwise constraints (constitutive laws)~\cite{Tartar1979,Tartar1983}. \textsc{Tartar}'s framework has been very useful in convex integration both in the Calculus of Variations \cite{Muller1999d,Muller2003}, as well as in fluid dynamics \cite{DeLellis2009,DeLellis2010}. Specific constitutive laws do not play a role in the proof of Corollary \ref{thm:Eulertheorem}, and in fact, an analogous result holds for numerous other incompressible models of fluid mechanics. The result also trivially extends to \emph{subsolutions}, that is solutions of the linear equations which take values in the so-called $\Lambda$-convex hull. Subsolutions can be interpreted as coarse-grained averages, see e.g.\ \cite{Castro2019,DeLellis2012}.

Corollary \ref{thm:Eulertheorem} follows immediately from the next lemma, using interpolation and the fact that the class of residual $G_\delta$ sets is closed under countable intersections.

\begin{lemma} \label{lemma:Eulerlemma}
Let $n \ge 2$, $M > 0$ and $p \in (2,\infty)$. The set of data $u^0 \in L^2$ with a solution $u \in M \mathbb{B}_{L^p_t L^2_{\sigma,x}}$ of \eqref{eq:relaxedEuler} is nowhere dense in $L^2_\sigma$.
\end{lemma}

\begin{proof}
Denote $D = \{(u,S) \in L^p_t L^2_{\sigma,x} \times L^{p/2}_t \mathbf{M}_x \colon \eqref{eq:relaxedEuler} \text{ holds for some } u^0 \in L^2_\sigma\}$ and define {$T \colon D \to L^2_\sigma$} by $T(u,S) \equiv u^0$. Our intention is to verify the assumptions of Theorem \ref{thm:linearproblems}.

Let $(u,S) \in D$. Given $\lambda > 0$ we set
\begin{equation} \label{eq:scaled u}
u_{\lambda}(x,t) \equiv u \left( \frac{x}{\lambda}, \frac{t}{\lambda} \right), \quad
S_{\lambda}(x,t) \equiv S \left( \frac{x}{\lambda}, \frac{t}{\lambda} \right), \quad
u^0_{\lambda}(x,t) \equiv u^0 \left( \frac{x}{\lambda} \right). 
\end{equation}
Now \eqref{eq:relaxedEuler}--\eqref{eq:scaled u} imply that $(u_\lambda,S_\lambda) \in D$ and $T(u_\lambda,S_\lambda) = u^0_\lambda$. We compute 
\[\Vert u_\lambda \Vert_{L^p_t L^2_x} = \lambda^{\frac{n}{2} + \frac{1}{p}} \Vert u \Vert_{L^p_t L^2_x}, \quad
\Vert S_\lambda \Vert_{L^{p/2}_t \mathbf{M}_x} = \lambda^{n + \frac{2}{p}} \Vert S \Vert_{L^{p/2}_t \mathbf{M}_x}, \quad
\Vert u^0_\lambda \Vert_{L^2} = \lambda^{\frac{n}{2}} \Vert u^0 \Vert_{L^2}.\]
Again we set $\tilde{\tau}_\lambda^D = \lambda \, \tp{id}$ and $\tilde{\tau}_\lambda^{Y^*} = \lambda \, \tp{id}$; Theorem \ref{thm:linearproblems} implies the claim.
\end{proof}

We conclude this subsection by briefly comparing Corollary \ref{thm:Eulertheorem} with the existing literature and we focus on the case $n=2$, where the picture is more complete. Following \cite{DeLellis2010}, we say that an initial datum $u^0$ is \textit{wild} if \eqref{eq:Euler1}--\eqref{eq:Euler3} admits infinitely many admissible weak solutions. Combining the results of \cite{Szekelyhidi2012} with \cite[Theorem 4.2]{Lions1996}, we arrive at the following:

\begin{thm}
When $n=2$, the set of wild initial data is a dense, meagre $F_\sigma$ subset of $L^2_\sigma$.
\end{thm}

We also note that some wild initial data admits compactly supported solutions \cite{DeLellis2010}, while Corollary \ref{thm:Eulertheorem} shows that such solutions exist only for a meagre $F_\sigma$ set of initial data.

\subsection{Energy decay rate in the Navier--Stokes equations}
We also illustrate the use of Theorem \ref{thm:linearproblems} in the presence of viscosity; we use the Navier--Stokes equations in $\R^n \times [0,\infty)$, $n \ge 2$, as an example. Given an initial datum $u^0 \in L^2_\sigma$, recall that weak solutions of \eqref{eq:NS1}--\eqref{eq:NS3} were defined in $L^\infty_t L^2_{\sigma,x} \cap L^2_t \dot{H}^1_x$ in \textsection \ref{sec:Examples}. Furthermore, a weak solution is called a \emph{Leray--Hopf solution} if it satisfies the energy inequality
\[\frac{1}{2} \int_{\R^3} |u(x,t)|^2 \d x + \nu \int_s^t \int_{\R^3} |\D u(x,\tau)|^2 \d x \d \tau \le \frac{1}{2} \int_{\R^3} |u(x,s)|^2 \d x \qquad \tp{for all } t > s\]
for a.e.\ $s \in [0,\infty)$, including $s = 0$. \textsc{Leray} showed in his milestone paper \cite{Leray1934} that for every initial datum $u^0 \in L^2_\sigma$ there exists a Leray--Hopf solution {$u \in L^\infty_t L^2_{\sigma,x} \cap L^2_t \dot{H}^1_x$ with $u(\cdot,0) = u^0$}. We briefly recall some of the pertinent results on energy decay of Leray--Hopf solutions and refer to the recent review~\cite{Brandolese2018} for more details and references.

\textsc{Leray} asked in~\cite{Leray1934} whether $\mathcal{E}(t) = \frac 12 \int_{\R^3} |u(x,t)|^2 \d x \to 0$ as $t \to \infty$ for all Leray--Hopf solutions. An affirmative answer was given by theorems of \textsc{Kato} and \textsc{Masuda}, see~\cite[Theorem 2--3]{Brandolese2018}. \textsc{Schonbeck} has shown that there is no uniform energy decay rate for general data $u^0 \in L^2_\sigma$; more precisely, for every $\beta,\e,T > 0$ there exists $u^0 \in \beta \mathbb{B}_{L^2_\sigma}$ such that a Leray--Hopf solution satisfies $\mathcal{E}(T) \ge (1-\e) \mathcal{E}(0)$. Furthermore, whenever $u^0 \in L^2_\sigma \setminus \cup_{1 \le p < 2} L^p$, the energy $\mathcal{E}(t)$ does not undergo polynomial decay. Several precise statements on the decay rate of $\mathcal{E}(t)$ under extra integrability assumptions on $u^0 \in L^2_\sigma$ are given in~\cite{Brandolese2018}.

In Corollary \ref{cor:NSdecay} below, we recover the lack of polynomial decay for a Baire-generic datum. The result applies to all \emph{distributional solutions} of \eqref{eq:NS1}--\eqref{eq:NS3}, which we define as mappings {$u \in L^2_{\tp{loc},t} L^2_{\sigma,x}(\R^3 \times [0,\infty), \R^3)$} such that
\[\int_0^\infty \langle u, \partial_t \varphi \rangle \d t + \int_0^\infty \langle u \otimes u, \D \varphi \rangle \d t + \nu \int_0^\infty \langle u, \Delta \varphi \rangle \d t + \langle u^0, \varphi(0) \rangle = 0\]
for all $\varphi \in C_c^\infty(\R^3 \times [0,\infty),\R^3)$ with $\tp{div} \, \varphi = 0$.

\begin{prop} \label{prop:NSdecay}
Let $p \in (2,\infty)$ and $M > 0$. The set of initial data for which \eqref{eq:NS1}--\eqref{eq:NS3} admits a distributional solution $u \in M \mathbb{B}_{L^p_t L^2_x}$ is nowhere dense in $L^2_\sigma$.
\end{prop}

\begin{proof}
We consider the relaxed problem where we require $u \in L^p_t L^2_{\sigma,x}$, $S^1 \in L^{p/2}_t \mathbf{M}_x$ and {$S^2 \in L^p_t \dot{H}^{-1}_x$} to satisfy
\begin{equation} \label{eq:relaxedNS}
\int_0^\infty \langle u, \partial_t \varphi \rangle \d t + \int_0^\infty \langle S^1, \D \varphi \rangle \d t + \nu \int_0^\infty \langle S^2, \D \varphi \rangle \d t + \langle u^0, \varphi(0) \rangle = 0
\end{equation}
for all $\varphi \in C_c^\infty([0,\infty),\R^3)$ with $\tp{div} \, \varphi = 0$. As before, denote by 
$$D \subset L^p_t L^2_{\sigma,x} \oplus L^{p/2}_t \mathbf{M}_x \oplus L^p_t \dot{H}^{-1}_x \equiv X^*$$ the set of triples $(u,S^1,S^2)$ such that \eqref{eq:relaxedNS} holds. One sets
\[u_\lambda(x,t) = u \left( \frac{x}{\lambda}, \frac{t}{\lambda} \right), \qquad S^i(x,t) = S^i \left( \frac{x}{\lambda}, \frac{t}{\lambda} \right);\]
note that
\begin{align*}
& \Vert u_\lambda \Vert_{L^p_t L^2_{x}} = \lambda^{n/2+1/p} \Vert u\Vert_{L^p_t L^2_x},  
& \Vert S^1_\lambda \Vert_{L^{p/2}_t \mathbf{M}_x} = \lambda^{n+2/p} \Vert S^1 \Vert_{L^{p/2}_t \mathbf{M}_x}, \\
& \Vert u^0_\lambda \Vert_{L^2} = \lambda^{n/2} \Vert u \Vert_{L^2},  & \Vert S^2_\lambda \Vert_{L^p_t \dot{H}^{-1}_x} = \lambda^{n/2+1+2/p} \Vert S^2 \Vert_{L^p_t \dot{H}^{-1}_x}.
\end{align*}
As before, we set $\tilde{\tau}_\lambda^D =  \lambda \, \tp{id}$ and $\tilde{\tau}_\lambda^{Y^*} = \lambda \, \tp{id}$. The claim now follows from Theorem \ref{thm:linearproblems}.
\end{proof}

\begin{cor} \label{cor:NSdecay}
Let $\e,C > 0$. Consider the set $X_{C,\e}$ of initial data $u^0 \in L^2_\sigma$ such that a distributional solution of \eqref{eq:NS1}--\eqref{eq:NS3} satisfies
\[(1+t)^\e \mathcal{E}(t) \le C \qquad \text{for almost every } t \in [0,\infty).\]
The set $X_{C,\e}$ is nowhere dense in $L^2_\sigma$.
In particular, for a Baire-generic $u^0 \in L^2_\sigma$, distributional solutions satisfy
\[\lim_{\tau \to \infty} \Vert t^\e \mathcal{E} \Vert_{L^\infty(\tau,\infty)} = \infty \quad \text{for every } \e > 0.\]
\end{cor}

\section{Concluding discussion} \label{sec:Concluding discussion}
In this section, we discuss the advantages of the nonlinear open mapping principles proved in this paper, when compared to the classical Banach--Schauder theorem. We also point out some of the limitations of our results, as well as directions for future work.

We begin by recalling the standard proof of the Banach--Schauder open mapping theorem. If a bounded \textit{linear} map $L \colon X \to Y$ between Banach spaces is surjective, then the Baire category theorem yields a constant $C > 0$ and a ball $B_r(f_0) \subset Y$ such that $L(C \mathbb{B}_X) \supseteq B_r(f_0)$, and the proof is completed as follows. First, by linearity, $L(C\mathbb{B}_X) = L(-C \mathbb{B}_X) \supseteq - B_r(f_0)$, so that, by linearity again, 
$$L(2 C \mathbb{B}_X) = L(C \mathbb{B}_X) - L(C \mathbb{B}_X) \supset B_r(f_0) - B_r(f_0) = 2 B_r(0).$$
We notice that this proof uses in a fundamental way three properties: 
\begin{enumerate}
\item\label{it:linearity} the linearity of the operator $L$;
\item\label{it:vectorspace} the vector space structure of the domain of definition of $L$;
\item\label{it:symrange} the symmetry of the range of $L$.
\end{enumerate}

Concerning \ref{it:linearity}, we note that if one attempts to generalise the above proof to nonlinear operators, then surjectivity only leads to ``$1/2$-openness'' and, more generally, $1/n$-surjectivity leads to $1/2n$-openness. To our knowledge, Theorem \ref{thm:abstractopenmappingprinciple} and Proposition \ref{prop:trichotomy} give the first abstract results on \textsc{Rudin}'s problem which yield $1/n$-openness from $1/n$-surjectivity.

With respect to \ref{it:vectorspace}, another key novelty of our work is that the domain of definition $D$ of the operator $T$ need not be a vector space. This is crucial when applying open mapping theorems to typical Cauchy problems in nonlinear evolutionary PDEs as is done in \textsection \ref{sec:A general nonlinear open mapping principle for scale-invariant problems}--\ref{sec:Non-surjectivity under incompatible scalings}.

Finally, we note that \ref{it:symrange} is not needed for our results either. In fact, Theorems \ref{thm:abstractopenmappingprinciple} and \ref{thm:abstractopenmappingprinciple2} apply when the target space is a closed convex cone such as $\{f \in L^p(\R^n) \colon f \ge 0 \tp{ a.e.}\}$, for $1 < p < \infty$, c.f.\ Remark \ref{remark:cones}, and also when the symmetry of the range is non-trivial to check, as is the case for the Hessian operator $\tp{H} \colon \dot{W}^{1,2}(\R^2) \to \mathscr{H}^1(\R^2)$.

We now discuss some of the limitations of our work. From a PDE perspective, the main weak point of Theorem \ref{thm:abstractopenmappingprinciple2}  is that assumption \ref{it:isometries2} seems difficult to adapt to function spaces defined over the flat torus $\T^n$ or bounded domains. For instance, on $\T^n$, translations $\sigma_k^{Y^*} f(x) = f(x-ke)$ typically fail the condition $\sigma_k^{Y^*} f \overset{*}{\rightharpoonup} 0$. On $\R^n$, translations can often be replaced by scalings such as $\sigma_k^{X^*} f(x) = k^\alpha f(k x)$, but such operators are of course not automorphisms on function spaces over the torus or bounded domains.

Despite the fact that Theorem \ref{thm:abstractopenmappingprinciple2} applies to many different equations, it would be interesting to look for generalisations, in order to account for other physical PDEs. 
One such situation concerns  function spaces with \textit{critical scalings}, i.e.\ 
$$\Vert \tau_\lambda^D u \Vert_{X^*} = \Vert u \Vert_{X^*}\qquad \tp{ and } \qquad \Vert \tau_\lambda^{Y^*} f \Vert_{Y^*} = \Vert f \Vert_{Y^*}.$$ Note, however, that even if one does not assume that $T$ is positively homogeneous, as in Theorem \ref{thm:abstractopenmappingprinciple}, the proof of this theorem still provides $\delta,M_\delta > 0$ such that $T(M_\delta \mathbb{B}_{X^*}) \supseteq \delta \mathbb{B}_{Y^*}$. It thus seems natural to ask whether one can achieve openness at the origin, i.e., whether one gets $\lim_{\delta \searrow 0} M_\delta = 0$.
Another interesting problem is to decide whether the weak$^*$-to-weak$^*$ closed graph assumption on the operators is an artifact of our proofs or a fundamental requirement for the validity of a nonlinear open mapping principle. 
We hope to address these questions in future work.

\let\oldthebibliography\thebibliography
\let\endoldthebibliography\endthebibliography
\renewenvironment{thebibliography}[1]{
  \begin{oldthebibliography}{#1}
    \setlength{\itemsep}{1.5pt}
    \setlength{\parskip}{1.5pt}
}
{
  \end{oldthebibliography}
}

{\small
\bibliographystyle{acm}
\bibliography{/Users/antonialopes/Dropbox/Oxford/Bibtex/library.bib}
}

\end{document}